\documentclass{amsart}
\usepackage{amsmath,amsthm,amssymb}
\usepackage{hyperref}
\usepackage{pstricks,pst-node,multido,pst-coil}

\usepackage{a4wide}

\title{Enumeration formulas for generalized $q$-Euler numbers}
\author{Jang Soo Kim}
\email{kimjs@math.umn.edu}

\newtheorem{thm}{Theorem}[section]
\newtheorem{lem}[thm]{Lemma}
\newtheorem{prop}[thm]{Proposition}
\newtheorem{cor}[thm]{Corollary}
\theoremstyle{definition}

\newtheorem{defn}{Definition}

\newtheorem{problem}{Problem}
\theoremstyle{remark}

\def\leftbrace(#1,#2)#3[#4]{
\rput(-.3,0){\rput(#1,#2){\psscaleboxto(.5,#3){\{}}}
\rput(-.8,0){\rput[r](#1,#2){#4}}
}
\def\rightbrace(#1,#2)#3[#4]{
\rput(.3,0){\rput(#1,#2){\psscaleboxto(.5,#3){\}}}}
\rput(.8,0){\rput[l](#1,#2){#4}}
}
\def\upbrace(#1,#2)#3[#4]{
\rput(0,.3){\rput(#1,#2){\psscaleboxto(#3,.5){\rotateleft{\}}}}}
\rput(0,.8){\rput[b](#1,#2){#4}}
}
\def\downbrace(#1,#2)#3[#4]{
\rput(0,-.3){\rput(#1,#2){\psscaleboxto(#3,.5){\rotateleft{\{}}}}
\rput(0,-.8){\rput[t](#1,#2){#4}}
}
\newcommand{\dist}{\operatorname{dist}}

\newcommand{\Qbinom}[3]{\genfrac{[}{]}{0pt}{}{#1}{#2}_{#3}}
\newcommand{\qbinom}[2]{\genfrac{[}{]}{0pt}{}{#1}{#2}_q}

\newcommand\flr[1]{\left\lfloor #1\right\rfloor}

\newcommand\Qint[2]{\left[ #1\right]_{#2}}
\newcommand\qint[1]{\left[ #1\right]_q}

\newcommand\A{\mathcal{A}}
\newcommand\B{\mathcal{B}}

\newcommand\U{\mathcal{U}}
\newcommand\V{\mathcal{V}}
\newcommand\D{\mathcal{D}}

\newcommand\wt{\operatorname{wt}}
\newcommand\ha{\operatorname{h}}

\newcommand\one{\mathbf{1}}

\newcommand\norm[1]{\lVert #1\rVert}

\newcommand\tot{\textrm{31-2}}
\newcommand\aep{\alpha_\epsilon}
\newcommand\bep{\beta_\epsilon}

\def\dk{\delta_k}
\def\dkc.{$\delta_k$-configuration}
\def\dkpc.{$\delta_k^+$-configuration}
\def\dkmc.{$\delta_{k}^-$-configuration}

\def\Dp#1{\Delta_{#1}^+}
\def\Dm#1{\Delta_{#1}^-}
\def\ywt.{$(y,q)$-weight}

\newcommand\MD{\overline{\mathcal{D}}}

\newcommand\sop{\mathcal{SOP}}
\newcommand\tr{\mathrm{tr}}
\newcommand\mk{\mathrm{mark}}

\newcommand\diag{\mathrm{diag}}

\def\JV.{Josuat-Verg\`es}
\def\sch.{Schr\"oder}
\psset{gridlabels=0pt,subgriddiv=1, gridwidth=.1pt}
\psset{linewidth=.1pt, unit=.6cm}
\psset{unit=12pt}
\newcount\ax \newcount\ay
\newcount\bx \newcount\by
\newcount\cx \newcount\cy
\newcount\dx \newcount\dy
\def\cell(#1,#2)[#3]{
\ax=#2 \ay=#1
\multiply\ay by-1
\bx=\ax \by=\ay
\cx=\ax \cy=\ay
\dx=\ax \dy=\ay
\advance\bx by-1
\advance\dy by1
\advance\cx by-1
\advance\cy by1
\psline (\dx,\dy)(\ax,\ay)(\bx,\by)
\rput(\number\cx.5,
\ifnum\cy=0 -0.5\else\number\cy.5\fi){#3}
}
\def\mcell(#1,#2){
\rput(-.5,.5){\pscircle(#2,-#1){.2}}
}
\def\smcell(#1,#2){
\rput(-.5,.5){\pscircle*(#2,-#1){.2}}
}
\def\gcell(#1,#2)[#3]{
\ax=#2 \ay=#1
\multiply\ay by-1
\bx=\ax \by=\ay
\cx=\ax \cy=\ay
\dx=\ax \dy=\ay
\advance\bx by-1
\advance\dy by1
\advance\cx by-1
\advance\cy by1
\psframe[linestyle=none,fillstyle=solid,fillcolor=gray!40!white](\ax,\ay)(\cx,\cy)
\rput(\number\cx.5,
\ifnum\cy=0 -0.5\else\number\cy.5\fi){#3}
}
\def\Gcell(#1,#2)[#3]{
\gcell(#1,#2)[#3] \cell(#1,#2)[#3]
}
\def\psrow(#1,#2){\multido{\i=1+1}{#2}{\cell(#1,\i)[]}}
\def\psdk#1{
\psline(#1,0)(0,0)(0,-#1)
\multido{\n=1+1}{#1}{
  \ax=#1 \advance\ax by 1
  \advance\ax by -\n
  \psrow(\n,\ax)}
}
\def\rowseg#1[#2]{\rput(0,-#1){\rput(0,.5){
      \psline[linewidth=1.5pt](0,0)(#2,0)}}}
\def\rowarr#1[#2]{\rput(0,-#1){\rput(0,.5){
      \psline[linewidth=1.5pt,arrowsize=.5, arrowlength=.6]{->}(0,0)(#2,0)}}}
\def\colseg#1[#2]{\rput(#1,0){\rput(-.5,0){
      \psline[linewidth=1.5pt](0,0)(0,-#2)}}}
\def\colarr#1[#2]{\rput(#1,0){\rput(-.5,0){
      \psline[linewidth=1.5pt,arrowsize=.5, arrowlength=.6]{->}(0,0)(0,-#2)}}}
\def\harrow(#1,#2)[#3]{\rput(#2,-#1){\rput(-1,.5){
      \psline[linewidth=1.5pt,arrowsize=.5, arrowlength=.6]{->}(0,0)(#3,0)}}}
\def\varrow(#1,#2)[#3]{\rput(#2,-#1){\rput(-.5,1){
      \psline[linewidth=1.5pt,arrowsize=.5, arrowlength=.6]{->}(0,0)(0,-#3)}}}
\def\hzero(#1,#2){\rput(#2,-#1){\rput(-1,.5){\psarc*(0,0){.2}{-90}{90}}}}
\def\vzero(#1,#2){\rput(#2,-#1){\rput(-.5,1){\psarc*(0,0){.2}{180}{0}}}}
\def\hzerow(#1,#2){\rput(#2,-#1){\rput(-1,.5){\psarc(0,0){.2}{-90}{90}}}}
\def\vzerow(#1,#2){\rput(#2,-#1){\rput(-.5,1){\psarc(0,0){.2}{180}{0}}}}

\def\UP(#1,#2)[#3]{\rput(#1,#2){\psline(0,0)(1,1) 
}}
\def\DW(#1,#2)[#3]{\rput(#1,#2){\psline(0,0)(1,-1) 
}}
\def\MUP(#1,#2)[#3]{\rput(#1,#2){\psline[linewidth=2pt](0,0)(1,1)
}}
\def\MDW(#1,#2)[#3]{\rput(#1,#2){\psline[linewidth=2pt](0,0)(1,-1)
}}

\begin{document}

\begin{abstract}
  We find an enumeration formula for a $(t,q)$-Euler number which is a
  generalization of the $q$-Euler number introduced by Han, Randrianarivony, and
  Zeng. We also give a combinatorial expression for the $(t,q)$-Euler number and
  find another formula when $t=\pm q^r$ for any integer $r$. Special cases of
  our latter formula include the formula of the $q$-Euler number recently found
  by Josuat-Verg\`es and Touchard-Riordan's formula.
\end{abstract}

\maketitle

\section{Introduction}
\label{sec:introduction}

The \emph{Euler number} $E_n$ is defined by 
\[ 
\sum_{n\geq0}E_n\frac{x^n}{n!} = \sec x + \tan x.
\] 
Thus $E_{2n}$ and $E_{2n+1}$ are also called the \emph{secant number} and the
\emph{tangent number} respectively. In 1879, Andr\'e \cite{Andre1879} showed
that $E_{n}$ is equal to the number of \emph{alternating permutations} of
$\{1,2,\ldots,n\}$, i.e., the permutations $\pi=\pi_1\ldots \pi_{n}$ such that
$\pi_1<\pi_2 > \pi_3 < \cdots$.

There are several $q$-Euler numbers studied in the literature, for instance, see
\cite{Foata2010, Han1999, Huber2010, JV2010, Shin2010}. In this paper we
consider the following $q$-Euler number $E_n(q)$ introduced by Han et
al. \cite{Han1999}:
\begin{equation}
  \label{eq:2}
  \sum_{n\geq0} E_{2n}(q) x^n = \cfrac{1}{
    1- \cfrac{\qint{1}^2 x}{
       1- \cfrac{\qint{2}^2 x}{ \cdots}}}, \qquad
 \sum_{n\geq0} E_{2n+1}(q) x^n = \cfrac{1}{
    1- \cfrac{\qint{1}\qint2 x}{
       1- \cfrac{\qint{2}\qint3 x}{ \cdots}}},
\end{equation}
where $\qint{n}=(1-q^n)/(1-q)$. We will use the standard notations:
\[
(a;q)_n = (1-a)(1-aq)\cdots(1-aq^{n-1}), \qquad
\qbinom{n}{k} = \frac{(q;q)_n}{(q;q)_{k} (q;q)_{n-k}}.
\] 
 This $q$-Euler number also has a nice
combinatorial expression found by Chebikin \cite{Chebikin2008}:
\[E_n(q) = \sum_{\pi\in \mathfrak{A}_n} q^{\tot(\pi)},\] where $\mathfrak{A}_n$
denotes the set of alternating permutations of $\{1,2,\ldots,n\}$ and
$\tot(\pi)$ denotes the number of $\tot$ patterns in $\pi$.  

Recently, \JV. \cite{JV2010} found a formula for $E_n(q)$. In
Section~\ref{sec:jv.-orig-form} we show that, by elementary manipulations, his
formula can be rewritten as follows:
 \begin{align}
E_{2n}(q) &= \frac{1}{(1-q)^{2n}} 
\sum_{k=0}^n \left( \binom{2n}{n-k} - \binom{2n}{n-k-1} \right)
 q^{k(k+1)}\sum_{i=-k}^{k} (-q)^{-i^2},   \label{eq:seq}\\
 E_{2n+1}(q) &= \frac{1}{(1-q)^{2n}} 
 \sum_{k=0}^n \left( \binom{2n}{n-k} - \binom{2n}{n-k-1} \right)
 q^{k(k+2)} A_k(q^{-1}),
\label{eq:tan}
 \end{align}
where $A_0(q)=1$ and for $k\geq1$,
\[
A_k(q) =  \frac{1}{1-q} \sum_{i=-k}^{k}(-q)^{i^2}
+ \frac{q^{2k+1}}{1-q} \sum_{i=-(k-1)}^{k-1}(-q)^{i^2}.
\]
Shin and Zeng \cite[Theorem~12]{Shin2010} found a parity-independent formula for
$E_n(q)$.

 We note that \eqref{eq:seq} is similar to the following formula of Touchard
 \cite{Touchard1952} and Riordan \cite{Riordan1975}:
\begin{equation}
  \label{eq:TR}
d_n = \frac{1}{(1-q)^n}
  \sum_{k=0}^n \left( \binom{2n}{n-k} - \binom{2n}{n-k-1} \right)
  (-1)^k q^{\frac{k(k+1)}2},
\end{equation}
where $d_n$ is defined by
\begin{equation}
  \label{eq:3}
  \sum_{n\geq0} d_n x^n = \cfrac{1}{
    1- \cfrac{\qint{1} x}{
       1- \cfrac{\qint{2} x}{ \cdots}}}.
\end{equation}

In this paper we introduce the \emph{$(t,q)$-Euler numbers} $E_n(t,q)$ defined
by
\begin{equation}
  \label{eq:Etq}
 \sum_{n\geq0} E_{n}(t,q) x^n = \cfrac{1}{
    1- \cfrac{\qint{1}\Qint{1}{t,q} x}{
       1- \cfrac{\qint{2}\Qint{2}{t,q} x}{\cdots}}},
\end{equation}
where $\Qint{n}{t,q} = (1-tq^n)/(1-q)$. Note that $(1-q)^{2n} E_n(0,q)= (1-\sqrt
q)^{2n} E_n(-1,\sqrt q)=(1-q)^n d_n$, $E_n(1,q)=E_{2n}(q)$, and
$E_n(q,q)=E_{2n+1}(q)$.  In fact $E_n(t,q)$ is a special case of the $2n$th
moment $\mu_{2n}(a,b;q)$ of Al-Salam-Chihara polynomials $Q_n(x)$ defined by the
recurrence
\[
2xQ_n(x) = Q_{n+1} + (a+b)q^n Q_n(x) + (1-q^n)(1-abq^{n-1}) Q_{n-1}(x),
\]
and the initial conditions $Q_{-1}(x)=0$ and $Q_0(x)=1$.  If $a=\sqrt{-qt}$ and
$b=-\sqrt{-qt}$, then the $2n$th moment $\mu_{2n}(a,b;q)$ satisfies $(1-q)^{2n}
E_n(t,q) = 2^{2n} \mu_{2n}(\sqrt{-qt}, -\sqrt{-qt};q)$.
\JV. \cite[Theorem~6.1.1 or Equation~46]{JV_PASEP} found a formula for
$\mu_n(a,b;q)$, which implies that
\begin{equation}
  \label{eq:JV_mu}
E_n(t,q) = \frac{1}{(1-q)^{2n}} \sum_{k=0}^n 
\left( \binom{2n}{n-k}-\binom{2n}{n-k-1} \right) 
\sum_{i,j\ge0} (-1)^{k+i}  q^{\binom{j+1}2} (qt)^{k-j} 
\qbinom{2k-j}{j} \qbinom{2k-2j}{i}. 
\end{equation}
In the same paper, \JV. showed that \eqref{eq:seq} and \eqref{eq:tan} can be
obtained from \eqref{eq:JV_mu} using certain summation formulas.  

The original motivation of this paper is to find a formula from which one can
easily obtain \eqref{eq:seq}, \eqref{eq:tan}, and \eqref{eq:TR}.  The main
results in this paper are Theorems~\ref{thm:main2} and \ref{thm:int} below.

\begin{thm}\label{thm:main2}
We have
\[ E_n(t,q) = \frac{1}{(1-q)^{2n}} \sum_{k=0}^n \left( \binom{2n}{n-k} -
  \binom{2n}{n-k-1} \right) t^k q^{k(k+1)} T_k(t^{-1}, q^{-1}),\] where
$\{T_k(t,q)\}_{k\ge0}$ is a family of polynomials in $t$ and $q$ determined
uniquely by the recurrence relation: $T_0(t,q)=1$ and for $k\geq1$,
\begin{equation}
  \label{eq:rec_Tk}
T_k(t,q) = T_{k-1}(t,q) + (1+t)(-q)^{k^2} + 
(1-t^2) \sum_{i=1}^{k-1} (-q)^{k^2-i^2} T_{i-1}(t,q).
\end{equation}
\end{thm}

From the recurrence of $T_k(t,q)$, we immediately get $T_k(-1,q)=1$ and
$T_k(1,q)=\sum_{i=-k}^k (-q)^{i^2}$, which imply \eqref{eq:TR} and
\eqref{eq:seq} respectively. Using certain weighted lattice paths satisfying the
same recurrence relation we obtain the following formula for $T_k(t,q)$.

\begin{cor}\label{cor:Tk}
We have
\begin{align*}
T_k(t,q)=\sum_{j=0}^k \sum_{i=0}^j (-1)^{j+i}
t^{2i} q^{j^2+i^2+i} \Qbinom{k-j}{i}{q^2}
\left( \Qbinom{k-i}{j-i}{q^2}
+ t \Qbinom{k-i-1}{j-i-1}{q^2}
\right).
\end{align*}
\end{cor}

As a consequence of the proof of Corollary~\ref{cor:Tk} we can express
$T_k(t,q)$ using what we call self-conjugate overpartitions, see
Theorem~\ref{thm:sop}.  This combinatorial expression allows us to find a
functional equation for $T_k(t,q)$ which gives a recurrence relation for
$T_k(\pm q^r,q)$, see Corollary~\ref{cor:rec_T}. Solving the recurrence
relation, we get the following formulas for $T_n(\pm q^r,q)$ for any integer
$r$.

\begin{thm}\label{thm:int}
 For $b\geq0$ and $k\geq1$, we have
\begin{align}
T_k(q^b,q) &= 
\sum_{i=0}^{k-1} \frac{q^{i(2k+1)}}{(q;q)_b} \Qbinom{b}{i}{q^2} 
\sum_{j=-(k-i)}^{k-i} (-q)^{j^2}
+\sum_{i=0}^{b-1}\frac{(q;q)_i}{(q;q)_b} q^{k(2k+2i+1)} 
\Qbinom{b-i-1}{k-1}{q^2}, \label{eq:++b} \\
T_k(-q^b,q) &= 
\sum_{i=0}^{k-1} \frac{q^{i(2k+1)}}{(-q;q)_b} \Qbinom{b}{i}{q^2} 
+\sum_{i=0}^{b-1}\frac{(-q;q)_i}{(-q;q)_b} q^{k(2k+2i+1)} 
\Qbinom{b-i-1}{k-1}{q^2}, \label{eq:-+b}
\end{align}
and for $b\geq1$ and $k\geq0$, we have
\begin{align}
T_k(q^{-b},q) &= \sum_{i=0}^{b-1} (q^{1-b};q)_i (-q)^{k(k-2b+2)+2i} 
\Qbinom{k+i-1}{i}{q^2}, \label{eq:+-b}\\
T_k(-q^{-b},q) &= 
\sum_{i=0}^{k-1} (-q^{1-b};q)_b (-q)^{i(2k-2b-i+2)}
\Qbinom{b+i-1}{i}{q^2} \notag\\
&\quad +(-q)^{k^2+2k-2kb} \sum_{i=0}^{b-1} 
(-q^{1-b};q)_i q^{2i} \Qbinom{k+i-1}{i}{q^2}. \label{eq:--b}
\end{align}
\end{thm}

Note that \eqref{eq:seq} and \eqref{eq:tan} follows immediately from
\eqref{eq:++b} when $b=0$ and $b=1$, and \eqref{eq:TR} from \eqref{eq:-+b} when
$b=0$.  When $t=-q$ and $t=-1/q$, we get simple formulas, see
Propositions~\ref{thm:-q} and \ref{thm:-1/q}. 

We note that it is possible to obtain another formula for $T_k(q^b,q)$ for a
positive integer $b$ from a result in \cite[Section~6]{Prodinger}, see
Section~\ref{sec:concluding-remarks}.

The rest of this paper is organized as follows. In
Section~\ref{sec:reformulation} we interprete $E_n(t,q)$ using \dkc.s introduced
in \cite{JVKim}. In Section~\ref{sec:proofs} we prove Theorem~\ref{thm:main2}
and Corollary~\ref{cor:Tk}. In Section~\ref{sec:self-conj-overp} we show that
$T_k(t,q)$ can be expressed as the sum of certain weights of symmetric
overpartitions. Using this expression we also find a functional equation for
$T_k(t,q)$. In Section~\ref{sec:another_formula} using the functional equation
obtained in the previous section we prove Theorem~\ref{thm:int} which is divided
into Corollaries~\ref{cor:(b,k)+}, \ref{cor:(b,k)-}, \ref{cor:(-b,k)+}, and
\ref{cor:(-b,k)-}. In Section~\ref{sec:jv.-orig-form} we show that the original
formula of $E_n(q)$ in \cite{JV2010} is equivalent to \eqref{eq:seq} and
\eqref{eq:tan}. In Section~\ref{sec:concluding-remarks} we propose some open
problems.

\section{Interpretation of $E_n(t,q)$ using \dkc.s}
\label{sec:reformulation}

In this section we interpretate $E_n(t,q)$ using \dkc.s introduced in
\cite{JVKim}. The idea is basically the same as in \cite{JVKim}.

\subsection{S-fractions and weighted lattice paths}

An \emph{S-fraction} is a continued fraction of the following form:
\[
\cfrac{1}{
  1- \cfrac{c_1 x}{
    1- \cfrac{c_2 x}{\cdots}}}.
\]
Thus all continued fractions appeared in the introduction are S-fractions. There
is a simple combinatorial interpretation for S-fractions using weighted Dyck
paths. In this subsection we will find formulas equivalent to
Theorem~\ref{thm:main2} using this combinatorial interpretation.

\begin{defn}
  A \emph{Dyck path} of length $2n$ is a lattice path from $(0,0)$ to $(2n,0)$
  in $\mathbb{N}^2$ consisting of up steps $(1,1)$ and down steps $(1,-1)$. We
  denote by $\D_n$ the set of Dyck paths of length $2n$.  A \emph{marked Dyck
    path} is a Dyck path in which each up step and down step may be marked.  We
  denote by $\MD_n$ the set of marked Dyck paths of length $2n$.  We also denote
  by $\MD^*_n$ the subset of $\MD_n$ consisting of the marked Dyck paths without
  marked peaks. Here, a \emph{marked peak} means a marked up step immediately
  followed by a marked down step.  Given two sequences $\A=(a_1,a_2,\dots)$,
  $\B=(b_1,b_2,\dots)$ and $p\in\MD_n$, we define the weight $\wt(p;\A,\B)$ to
  be the product of $a_h$ (resp.~$b_h$) for each non-marked up step
  (resp.~non-marked down step) between height $h$ and $h-1$.
\end{defn}

Observe that every marked step can be considered as a step of weight 1.  We will
consider a Dyck path as a marked Dyck path without marked steps.  In this
identification we have $\D_n\subset \MD_n$.

The following combinatorial interpretation of S-fractions is well-known, see
\cite{Flajolet1980}.

\begin{lem}\label{lem1}
For two sequences $\A=(a_1,a_2,\dots)$, $\B=(b_1,b_2,\dots)$, we have
\[
\cfrac{1}{
  1- \cfrac{a_1 b_1 x}{
    1- \cfrac{a_2 b_2 x}{\cdots}}} =
\sum_{n\ge0} x^n \sum_{ p \in \D_n} \wt(p;\A,\B).
\]
\end{lem}

The reader may have noticed that every formula in the introduction has the
factor $\binom{2n}{n-k} - \binom{2n}{n-k-1}$ in its summand. This can be
explained by the following lemma.

\begin{lem} \label{lem2}\cite[Lemma~1.2]{JVKim} 
For two sequences $\A$ and $\B$
  we have
\[
 \sum_{p\in \D_n} \wt(p;\A,\B)
=\sum_{k=0}^n \left( \binom{2n}{n-k} - \binom{2n}{n-k-1} \right) 
\sum_{p\in \MD^*_k} \wt(p;\A-\one,\B-\one),
\]
where, if $\A=(a_1,a_2,\dots)$, the sequence $\A-\one$ means
$(a_1-1,a_2-1,\dots)$.
\end{lem}

From now on we fix the following sequences:
\[
\U = (-q, -q^2,\dots), \qquad \V_t = (-tq, -tq^{2}, \ldots). 
\]

By Lemmas~\ref{lem1} and \ref{lem2}, we have
\begin{align}
E_{n}(t,q) &= \sum_{p\in \D_n} 
\wt(p;(\qint{1}, \qint{2}, \ldots), (\Qint{1}{t,q}, \Qint{2}{t,q}, \ldots))\notag\\
&=\frac{1}{(1-q)^{2n}} \sum_{p\in \D_n} 
\wt(p;(1-q, 1-q^{2}, \ldots),(1-tq, 1-tq^{2}, \ldots))\notag\\
&=\frac{1}{(1-q)^{2n}} \sum_{k=0}^n \left( \binom{2n}{n-k} - \binom{2n}{n-k-1} \right) 
\sum_{p\in \MD^*_k} \wt(p;\U,\V_t).   \label{eq:m2}
\end{align}

\subsection{ \dkpc.s}

We now recall \dkc.s.  We first need some terminologies on integer partitions.

\begin{defn}
  A \emph{partition} is a weakly decreasing sequence
  $\lambda=(\lambda_1,\lambda_2,\dots,\lambda_\ell)$ of positive integer.
  Sometime we will consider that infinitely many zeros are attached at the end
  of $\lambda$ so that $\lambda_i=0$ for all $i>\ell$. Each integer $\lambda_i$
  is called a \emph{part} of $\lambda$ and the size of $\lambda$, denoted
  $|\lambda|$, is the sum of all parts.  The \emph{Ferrers diagram} of $\lambda$
  is the arrangement of left-justified square cells in which the $i$th topmost
  row has $\lambda_i$ cells. We will identify a partition with its Ferrers
  diagram.  \emph{Row $i$} (resp.~\emph{Column $i$}) means the $i$th topmost row
  (resp.~leftmost column). The \emph{$(i,j)$-cell} means the cell in Row $i$ and
  Column $j$. An \emph{inner corner} (resp.~\emph{outer corner}) of $\lambda$ is
  a cell $c\in\lambda$ (resp.~$c\in\dk/\lambda$) such that
  $\lambda\setminus\{c\}$ (resp.~$\lambda\cup\{c\}$) is a partition.  For a
  partition $\lambda$, the \emph{transpose} (or \emph{conjugate}) of $\lambda$
  is the partition, denoted $\lambda^{\tr}$, such that $\lambda^{\tr}$ has the
  $(i,j)$-cell if and only if $\lambda$ has the $(j,i)$-cell.  For two partition
  $\lambda$ and $\mu$ we write $\mu\subset\lambda$ if the Ferrers diagram of
  $\mu$ is contained in that of $\lambda$. In this case we denote their
  difference as sets by $\lambda/\mu$.
\end{defn}

Let $\delta_k$ denote the staircase partition $(k,k-1,\dots,1)$.  Let $B(m,n)$
denote the box with $m$ rows and $n$ columns, that is, $B(m,n) =
(\overbrace{n,n,\dots,n}^m)$. It is well-known, for instance see
\cite{Stanley1997}, that
\begin{equation}
  \label{eq:box}
  \sum_{\lambda\subset B(m,n)} q^{|\lambda|} = \qbinom{m+n}{m}.
\end{equation}

\begin{defn}
  A \emph{\dkc.}  is a pair $(\lambda,A)$ of a partition
  $\lambda\subset\delta_{k-1}$ and a set $A$ of arrows each of which occupies a
  whole row or a whole column of $\dk/\lambda$ or $\delta_{k-1}/\lambda$. If an
  arrow occupies a whole row or a whole column of $\dk/\lambda$
  (resp.~$\delta_{k-1}/\lambda$), we call the arrow a $k$-arrow
  (resp.~$(k-1)$-arrow).  The \emph{length} of an arrow is the number of cells
  occupied by the arrow.  A \emph{fillable corner} is an outer corner which is
  occupied by one $k$-arrow and one $(k-1)$-arrow. A \emph{forbidden corner} is
  an outer corner which is occupied by two $k$-arrows.  A \emph{\dkpc.} is a
  \dkc. without forbidden corners nor $(k-1)$-arrows.
\end{defn}

We note that an arrow in a \dkc. can have length $0$. We will represent an arrow
of length $0$ as a half dot as shown in Figure~\ref{fig:dkconfig}.

\begin{figure}
  \centering
  \begin{pspicture}(0,0)(5,-5)
    \gcell(1,1)[] \gcell(1,2)[] \gcell(1,3)[] \gcell(1,4)[]
    \gcell(2,1)[] \gcell(2,2)[]
    \psdk5 \vzero(2,4) \varrow(2,3)[1] \harrow(3,1)[3] \harrow(4,1)[1]
  \end{pspicture}
 \caption{An example of \dkc..}
  \label{fig:dkconfig}
\end{figure}

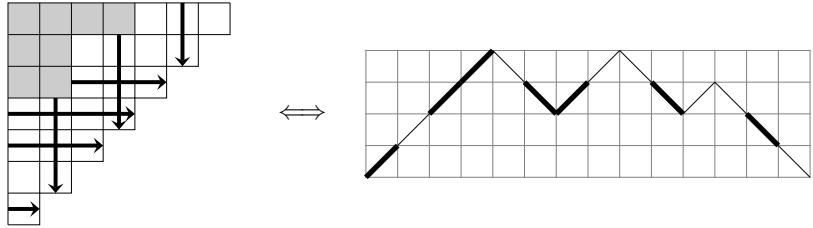
\begin{figure}
 \centering
\begin{pspicture}(0,0)(7,-7)
\gcell(1,1)[]\gcell(1,2)[]\gcell(1,3)[]\gcell(1,4)[]
\gcell(2,1)[]\gcell(2,2)[]
\gcell(3,1)[]\gcell(3,2)[]
\psdk7
\harrow(4,1)[4] \harrow(5,1)[3] \harrow(7,1)[1] \harrow(3,3)[3] 
\varrow(4,2)[3] \varrow(2,4)[3] \varrow(1,6)[2]
\end{pspicture} 
\begin{pspicture}(-4,-1.5)(14,5)
\psgrid[gridcolor=gray](0,0)(14,4)
\rput(-2,2){$\Longleftrightarrow$}
\MUP(0,0)[] \UP(1,1)[$-q^2$] \MUP(2,2)[] \MUP(3,3)[]
\DW(4,4)[$-q^4$] \MDW(5,3)[] \MUP(6,2)[] \UP(7,3)[$-q^4$]
\DW(8,4)[$-q^4$] \MDW(9,3)[] \UP(10,2)[$-q^3$] \DW(11,3)[$-q^3$]
\MDW(12,2)[] \DW(13,1)[$-q$]
\end{pspicture}
\caption{A \dkpc. and the corresponding marked Dyck path, where the marked steps
  are the thicker steps.}
  \label{fig:dk}
\end{figure}

There is a natural bijection between $\MD^*_k$ and $\Dp{k}$ as follows. For
$(\lambda,A)\in\Dp{k}$, the north-west border of $\dk/\lambda$ defines a
marked Dyck path of length $2k$ where the marked steps correspond to the
segments on the border with arrows, see Figure~\ref{fig:dk}.

For a \dkc. $C=(\lambda,A)$, we define the weight $\wt_{t,q}(C)$ by
\[
\wt_{t,q}(C) = (-1)^{|A|} t^{\ha(A)} q^{2|\lambda| + \norm{A}},
\]
where $\norm{A}$ is the sum of the arrow lengths and $\ha(A)$ is the number of
horizontal arrows.  For example, if $C$ is the \dkc. in Figure~\ref{fig:dk}, we
have $\wt_{t,q}(C)=(-1)^7 t^4 q^{2\cdot 8 + 1+3+4+3+3+3+2} $. 

\begin{lem}
  Suppose that $C\in\Dp{k}$ corresponds to $p\in\MD^*_k$ in the bijection
  described above. Then we have
  \begin{equation}
\label{eq:uv2}
\wt(p;\U,\V_t) =t^k q^{k(k+1)}\wt_{t^{-1},q^{-1}}(C),
 \end{equation}
which implies that
\begin{equation}
\label{eq:uv2sum}
\sum_{p\in \MD^*_k} \wt(p;\U,\V_t) =
t^k q^{k(k+1)} \sum_{C\in\Dp{k}}\wt_{t^{-1},q^{-1}}(C).
\end{equation}
\end{lem}
\begin{proof}
  Let $C=(\lambda,A)$. By the construction of the bijection sending $C$ to $p$,
  it is easy to see that
\[
\wt(p;\U,\V_t)=\prod_{i=1}^k r(i) c(i),
\]
where $r(i)=-tq^{(k+1-i)-\lambda_i}$ if there is no horizontal arrow in Row $i$
and $r(i)=1$ otherwise, and $c(i) =-tq^{(k+1-i)-\lambda^\tr_i}$ if there is no
vertical arrow in Column $i$ and $c(i)=1$ otherwise.

Now consider $t^kq^{k(k+1)} \wt_{t^{-1},q^{-1}}(C)$.  By the identities
\[
t^kq^{k(k+1)} = \prod_{i=1}^k (-tq^{k+1-i})(-q^{k+1-i}),
\]
\[
2|\lambda| = \sum_{i=1}^k \lambda_i + \sum_{i=1}^k \lambda^\tr_i,
\]
and the fact that if there is an arrow in Row $i$ (resp.~Column $i$) then its
length is $(k+1-i)-\lambda_i$ (resp.~$(k+1-i)-\lambda^\tr_i$), it is easy to
check
\[
t^kq^{k(k+1)}\wt_{t^{-1},q^{-1}}(C) =\prod_{i=1}^k r(i) c(i),
\]
which finishes the proof.
\end{proof}

\section{Proofs of Theorem~\ref{thm:main2} and Corollary~\ref{cor:Tk}}
\label{sec:proofs}

From now on we denote
\[
T_k(t,q) = \sum_{C\in \Dp k} \wt_{t,q}(C).
\]
For brevity we will also write $T_k$ instead of $T_k(t,q)$. 

By \eqref{eq:m2} and \eqref{eq:uv2sum}, we have
\[ E_n(t,q) = \frac{1}{(1-q)^{2n}} \sum_{k=0}^n \left( \binom{2n}{n-k} -
  \binom{2n}{n-k-1} \right) t^k q^{k(k+1)} T_k(t^{-1}, q^{-1}).\] Thus in order
to prove Theorem~\ref{thm:main2}, it remains to prove the recurrence relation
\eqref{eq:rec_Tk}.  We need some results in \cite{JVKim}. We begin by defining a
set which is in bijection with $\Dp k$.

A \emph{miniature} of a \dkc. is the restriction of it to the $(k-i,i)$-cell,
the $(k-i,i+1)$-cell, and the $(k-i+1,i)$-cell for some $1\leq i\leq k-1$, where
any $(k-1)$-arrows in Column $i+1$ or Row $k-i+1$ are ignored.  For example, the
miniatures of the \dkc. in Figure~\ref{fig:dkconfig} are
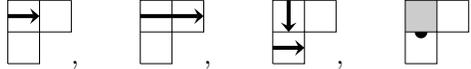
\begin{center}
  \begin{pspicture}(0,0)(2,-2)
\psdk2 \harrow(1,1)[1]
 \end{pspicture}, \qquad
 \begin{pspicture}(0,0)(2,-2)
\psdk2 \harrow(1,1)[2]
  \end{pspicture}, \qquad
  \begin{pspicture}(0,0)(2,-2)
\psdk2 \varrow(1,1)[1] \harrow(2,1)[1]
  \end{pspicture}, \qquad
\begin{pspicture}(0,0)(2,-2)
    \gcell(1,1)[]
\psdk2 \vzero(2,1)
  \end{pspicture},
\end{center}
where the bottommost miniature appears first.

\begin{defn}\label{def:dm}
  A \emph{\dkmc.} is a \dkc.  $(\lambda,A)$ satisfying the following conditions.
  \begin{enumerate}
  \item There is neither fillable corner nor forbidden corner.
  \item Every $k$-arrow has length $1$.
  \item For any miniature, if there is a horizontal (resp.~vertical) $k$-arrow
    in the bottom (resp.~right) cell, then the middle cell is contained in
    $\lambda$. Moreover, if the bottom (resp.~right) cell has a horizontal
    (resp.~vertical) $k$-arrow and a vertical (resp.~horizontal) $(k-1)$-arrow,
    then the right (resp.~bottom) cell has a horizontal (resp.~vertical)
    $k$-arrow. Pictorially, these mean the following:
    
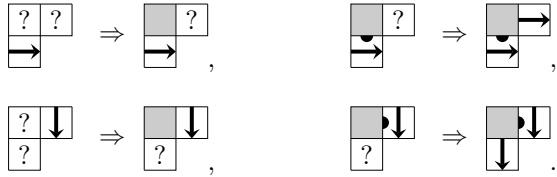
\begin{center}
   \begin{pspicture}(0,0)(2,-2)
       \cell(1,1)[$?$] \cell(1,2)[$?$] \cell(2,1)[] \psline(2,0)(0,0)(0,-2) \harrow(2,1)[1]
    \end{pspicture}
   \begin{pspicture}(-2,0)(2,-2)
\rput(-1,-1){$\Rightarrow$}
     \Gcell(1,1)[] \cell(1,2)[$?$] \cell(2,1)[]  \psline(2,0)(0,0)(0,-2) \harrow(2,1)[1]
    \end{pspicture},
   \begin{pspicture}(-4,0)(2,-2)
      \Gcell(1,1)[] \cell(1,2)[$?$] \cell(2,1)[] \psline(2,0)(0,0)(0,-2)
      \harrow(2,1)[1] \vzero(2,1)
   \end{pspicture}
   \begin{pspicture}(-2,0)(2,-2)
\rput(-1,-1){$\Rightarrow$}
      \Gcell(1,1)[] \cell(1,2)[] \cell(2,1)[]  \psline(2,0)(0,0)(0,-2)
      \harrow(2,1)[1] \vzero(2,1) \harrow(1,2)[1]
    \end{pspicture},

   \begin{pspicture}(0,0)(2,-2)
      \cell(1,1)[$?$] \cell(1,2)[] \cell(2,1)[$?$] \psline(2,0)(0,0)(0,-2) \varrow(1,2)[1]
    \end{pspicture}
   \begin{pspicture}(-2,0)(2,-2)
\rput(-1,-1){$\Rightarrow$}
       \Gcell(1,1)[] \cell(1,2)[] \cell(2,1)[$?$] \psline(2,0)(0,0)(0,-2) \varrow(1,2)[1]
    \end{pspicture},
   \begin{pspicture}(-4,0)(2,-2)
     \Gcell(1,1)[] \cell(1,2)[] \cell(2,1)[$?$] \psline(2,0)(0,0)(0,-2)
      \varrow(1,2)[1] \hzero(1,2)
   \end{pspicture}
   \begin{pspicture}(-2,1)(2,-2)
\rput(-1,-1){$\Rightarrow$}
      \Gcell(1,1)[] \cell(1,2)[] \cell(2,1)[] \psline(2,0)(0,0)(0,-2)
      \varrow(1,2)[1] \hzero(1,2) \varrow(2,1)[1]
    \end{pspicture}.
   \end{center}
 \end{enumerate}
 The set of \dkmc.s is denoted by $\Dm{k}$.
\end{defn}

\JV. and the author \cite[Proposition~4.1]{JVKim} found a bijection $\psi:\Dp
k\to\Dm k$ preserving $\wt_{1,q}$, i.e. $\wt_{1,q}(\psi(C))=\wt_{1,q}(C)$ for
all $C\in \Dp k$. From the construction of $\psi$ in their paper, it is clear
that $\psi$ also preserves the number of horizontal arrows. Thus we also have
$\wt_{t,q}(\psi(C))=\wt_{t,q}(C)$ for all $C\in \Dp k$, which implies
\begin{equation}
  \label{eq:Dp-Dm}
T_k = \sum_{C\in \Dm k} \wt_{t,q}(C).  
\end{equation}

Since $\Dp{k-1}\subset\Dm k$,  we can rewrite
\eqref{eq:Dp-Dm} as
\begin{equation}
  \label{eq:11}
T_k= T_{k-1} + \sum_{C\in \Dm k\setminus \Dp{k-1}} \wt_{t,q}(C).
\end{equation}

In order to compute the sum in \eqref{eq:11}, we need a property of the elements
in $\Dm k\setminus \Dp{k-1}$.

\begin{lem}\cite[Lemma~4.2]{JVKim} \label{lem:mini}
  Let $C\in \Dm k$. Then $C\in \Dm k \setminus \Dp {k-1}$ if and only if $C$ has
  a miniature listed in Figure~\ref{fig:list}.
\end{lem}

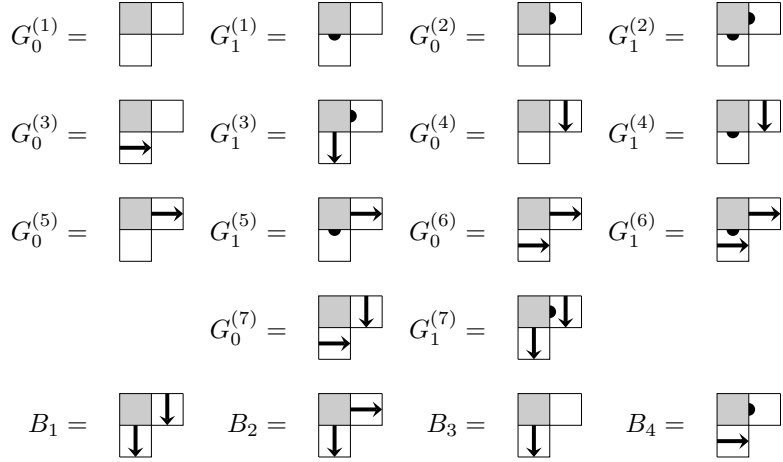
\begin{figure}
\begin{center}
\begin{pspicture}(-2,0)(4,-3)
\rput[r](-1,-1){$G^{(1)}_0=$}
\Gcell(1,1)[] \cell(2,1)[] \cell(1,2)[] \psline(2,0)(0,0)(0,-2)
\end{pspicture}
\begin{pspicture}(-2,0)(4,-3)
\rput[r](-1,-1){$G^{(1)}_1=$}
\Gcell(1,1)[] \cell(2,1)[] \cell(1,2)[] \psline(2,0)(0,0)(0,-2)
\vzero(2,1)
\end{pspicture}
\begin{pspicture}(-2,0)(4,-3)
\rput[r](-1,-1){$G^{(2)}_0=$}
\Gcell(1,1)[] \cell(2,1)[] \cell(1,2)[] \psline(2,0)(0,0)(0,-2)
\hzero(1,2)
\end{pspicture}
\begin{pspicture}(-2,0)(4,-3)
\rput[r](-1,-1){$G^{(2)}_1=$}
\Gcell(1,1)[] \cell(2,1)[] \cell(1,2)[] \psline(2,0)(0,0)(0,-2)
\hzero(1,2)
\vzero(2,1)
\end{pspicture}

\begin{pspicture}(-2,0)(4,-3)
\rput[r](-1,-1){$G^{(3)}_0=$}
\Gcell(1,1)[] \cell(2,1)[] \cell(1,2)[] \psline(2,0)(0,0)(0,-2)
\harrow(2,1)[1]
\end{pspicture}
\begin{pspicture}(-2,0)(4,-3)
\rput[r](-1,-1){$G^{(3)}_1=$}
\Gcell(1,1)[] \cell(2,1)[] \cell(1,2)[] \psline(2,0)(0,0)(0,-2)
\varrow(2,1)[1]
\hzero(1,2)
\end{pspicture}
\begin{pspicture}(-2,0)(4,-3)
\rput[r](-1,-1){$G^{(4)}_0=$}
\Gcell(1,1)[] \cell(2,1)[] \cell(1,2)[] \psline(2,0)(0,0)(0,-2)
\varrow(1,2)[1]
\end{pspicture}
\begin{pspicture}(-2,0)(4,-3)
\rput[r](-1,-1){$G^{(4)}_1=$}
\Gcell(1,1)[] \cell(2,1)[] \cell(1,2)[] \psline(2,0)(0,0)(0,-2)
\vzero(2,1)
\varrow(1,2)[1]
\end{pspicture}

\begin{pspicture}(-2,0)(4,-3)
\rput[r](-1,-1){$G^{(5)}_0=$}
\Gcell(1,1)[] \cell(2,1)[] \cell(1,2)[] \psline(2,0)(0,0)(0,-2)
\harrow(1,2)[1]
\end{pspicture}
\begin{pspicture}(-2,0)(4,-3)
\rput[r](-1,-1){$G^{(5)}_1=$}
\Gcell(1,1)[] \cell(2,1)[] \cell(1,2)[] \psline(2,0)(0,0)(0,-2)
\harrow(1,2)[1]
\vzero(2,1)
\end{pspicture}
\begin{pspicture}(-2,0)(4,-3)
\rput[r](-1,-1){$G^{(6)}_0=$}
\Gcell(1,1)[] \cell(2,1)[] \cell(1,2)[] \psline(2,0)(0,0)(0,-2)
\harrow(2,1)[1]\harrow(1,2)[1]
\end{pspicture}
\begin{pspicture}(-2,0)(4,-3)
\rput[r](-1,-1){$G^{(6)}_1=$}
\Gcell(1,1)[] \cell(2,1)[] \cell(1,2)[] \psline(2,0)(0,0)(0,-2)
\harrow(2,1)[1]\harrow(1,2)[1]  \vzero(2,1)
\end{pspicture}

\begin{pspicture}(-2,0)(4,-3)
\rput[r](-1,-1){$G^{(7)}_0=$}
\Gcell(1,1)[] \cell(2,1)[] \cell(1,2)[] \psline(2,0)(0,0)(0,-2)
\harrow(2,1)[1]\varrow(1,2)[1]
\end{pspicture}
\begin{pspicture}(-2,0)(4,-3)
\rput[r](-1,-1){$G^{(7)}_1=$}
\Gcell(1,1)[] \cell(2,1)[] \cell(1,2)[] \psline(2,0)(0,0)(0,-2)
\varrow(2,1)[1]\varrow(1,2)[1]  \hzero(1,2)
\end{pspicture}

\begin{pspicture}(-2,0)(4,-3)
\rput[r](-1,-1){$B_1=$}
\Gcell(1,1)[] \cell(2,1)[] \cell(1,2)[] \psline(2,0)(0,0)(0,-2)
\varrow(2,1)[1]\varrow(1,2)[1]
\end{pspicture}
\begin{pspicture}(-2,0)(4,-3)
\rput[r](-1,-1){$B_2=$}
\Gcell(1,1)[] \cell(2,1)[] \cell(1,2)[] \psline(2,0)(0,0)(0,-2)
\varrow(2,1)[1]\harrow(1,2)[1]
\end{pspicture}
\begin{pspicture}(-2,0)(4,-3)
\rput[r](-1,-1){$B_3=$}
\Gcell(1,1)[] \cell(2,1)[] \cell(1,2)[] \psline(2,0)(0,0)(0,-2)
\varrow(2,1)[1]
\end{pspicture}
\begin{pspicture}(-2,0)(4,-3)
\rput[r](-1,-1){$B_4=$}
\Gcell(1,1)[] \cell(2,1)[] \cell(1,2)[] \psline(2,0)(0,0)(0,-2)
\harrow(2,1)[1]
\hzero(1,2)
\end{pspicture}
\end{center}
 \caption{List of exceptions.}
  \label{fig:list}
\end{figure}

\begin{figure}
   \centering
\begin{pspicture}(0,0)(8,-8)
\gcell(1,1)[]  \gcell(1,2)[] \gcell(1,3)[] \gcell(1,4)[] \gcell(1,5)[]\gcell(1,6)[]
\gcell(2,1)[]  \gcell(2,2)[] \gcell(2,3)[] \gcell(2,4)[] 
\gcell(3,1)[]  \gcell(3,2)[] \gcell(3,3)[] \gcell(3,4)[]
\gcell(4,1)[]  \gcell(4,2)[] \gcell(4,3)[] \gcell(4,4)[]
\gcell(5,1)[]  \gcell(5,2)[] \gcell(5,3)[]
\gcell(6,1)[]
\gcell(7,1)[]
\psdk8
\harrow(1,7)[1] \harrow(4,5)[1]\varrow(5,4)[1]\vzero(6,3)
\harrow(7,2)[1] \harrow(8,1)[1]\varrow(2,5)[2]
\rput(-2,-2){\pspolygon[linewidth=1pt,linecolor=red](4,-2)(4,-4)(5,-4)(5,-3)(6,-3)(6,-2)}
\end{pspicture}
\begin{pspicture}(0,0)(4,-8)
\rput(2,-3){$\phi$}
\rput(2,-4){$\Longleftrightarrow$}
\end{pspicture}
\begin{pspicture}(0,0)(8,-8)
\gcell(1,1)[]  \gcell(1,2)[] \gcell(1,3)[] \gcell(1,4)[] \gcell(1,5)[]\gcell(1,6)[]
\gcell(2,1)[]  \gcell(2,2)[] \gcell(2,3)[] \gcell(2,4)[] 
\gcell(3,1)[]  \gcell(3,2)[] \gcell(3,3)[] \gcell(3,4)[]
\gcell(4,1)[]  \gcell(4,2)[] \gcell(4,3)[] \gcell(4,4)[]
\gcell(5,1)[]  \gcell(5,2)[] \gcell(5,3)[]
\gcell(6,1)[]
\gcell(7,1)[]
\psdk8
\harrow(1,7)[1] \harrow(4,5)[1]\varrow(5,4)[1]
\harrow(7,2)[1] \harrow(8,1)[1] \varrow(2,5)[2]
\rput(-2,-2){\pspolygon[linewidth=1pt,linecolor=red](4,-2)(4,-4)(5,-4)(5,-3)(6,-3)(6,-2)}
\end{pspicture}
\caption{The sign-reversing involution $\phi$ on $\Dm k\setminus\Dp{k-1}$. The
  topmost good exception is colored red.}
   \label{fig:sign-reversing}
 \end{figure}
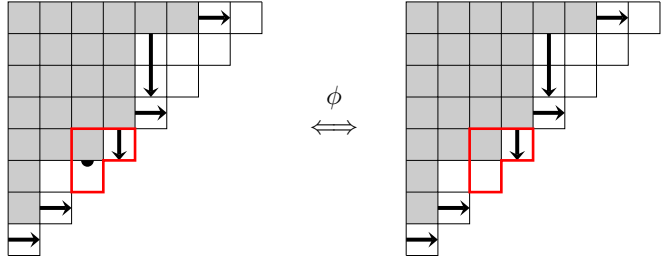

 We call the miniatures in Figure~\ref{fig:list} \emph{exceptions}. The exceptions
 $B_1$, $B_2$, $B_3$, and $B_4$ are called \emph{bad exceptions}, and the others
 are called \emph{good exceptions}.

Now we can compute the sum in \eqref{eq:11}. 

\begin{lem}\label{lem:rec}
We have
\[
\sum_{C\in \Dm k\setminus \Dp {k-1}} \wt_{t,q}(C)
=(1+t)(-q)^{k^2} + (1-t^2) \sum_{i=1}^{k-1} (-q)^{k^2-i^2} T_{i-1}.
\]
\end{lem}
\begin{proof}
  We will construct a sign-reversing involution $\phi$ on $\Dm k\setminus \Dp
  {k-1}$, i.e. an involution satisfying $\wt_{t,q}(\phi(C)) = - \wt_{t,q}(C)$
  if $\phi(C)\ne C$. If $\phi(C)= C$, we call $C$ a \emph{fixed point} of
  $\phi$.

  Suppose $C\in\Dm k\setminus \Dp {k-1}$. By Lemma~\ref{lem:mini}, $C$ has an
  exception. If $C$ has a good exception, find the topmost good exception. If
  the topmost good exception is $G_0^{(i)}$ (resp.~$G_1^{(i)}$) for some
  $i=1,2,\dots,7$, we define $\phi(C)$ to be the configuration obtained from $C$
  by replacing the topmost good exception with $G_1^{(i)}$ (resp.~$G_0^{(i)}$),
  see Figure~\ref{fig:sign-reversing}.  If $C$ has no good exceptions, we define
  $\phi(C)=C$. The map $\phi$ is certainly a sign-reversing involution whose
  fixed points are those containing only bad exceptions.

  Now suppose that $C$ has only bad exceptions. Note that the topmost bad
  exception determined the miniatures below it because the miniature below
  $B_1,B_2,B_3$ must be $B_1$ and the miniature below $B_4$ must be $B_2$.
  Furthermore, $B_1$ or $B_2$ can be the topmost exception only if it intersects
  with the first row. Thus $C$ looks like one of the configurations in
  Figure~\ref{fig:bad_exceptions}. Since there is no exception above the topmost
  bad exception, the sub-configurations consisting of ?'s in the last two
  configurations in Figure~\ref{fig:bad_exceptions} are contained in $\Dp
  {i-1}$, where $i$ can be any integer in $\{1,2,\dots,k-1\}$. Thus the weight
  sum of the configurations in Figure~\ref{fig:bad_exceptions} are, from left to
  right,
\[
(-q)^{k^2},\quad t(-q)^{k^2},\quad \sum_{i=1}^{k-1}(-q)^{k^2-i^2}T_{i-1}, \quad
\sum_{i=1}^{k-1}-t^2(-q)^{k^2-i^2}  T_{i-1}.
\]
Since the left hand side of the equation of the lemma is the weight sum of fixed
points of $\phi$, we are done. 
\end{proof}

  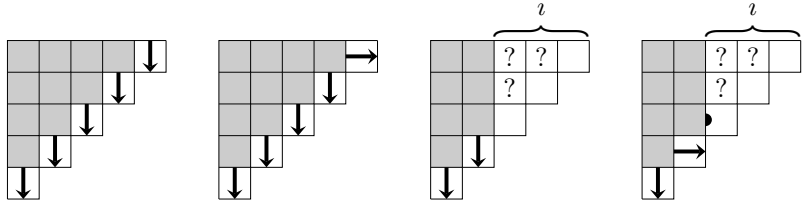
\begin{figure}
    \centering
    \begin{pspicture}(0,1.5)(5,-5) 
\gcell(1,1)[] \gcell(1,2)[] \gcell(1,3)[] \gcell(1,4)[]
\gcell(2,1)[] \gcell(2,2)[] \gcell(2,3)[] 
\gcell(3,1)[] \gcell(3,2)[] 
\gcell(4,1)[] 
\psdk5
\varrow(1,5)[1] \varrow(2,4)[1] \varrow(3,3)[1] \varrow(4,2)[1] \varrow(5,1)[1]
  \end{pspicture}\qquad
    \begin{pspicture}(0,1.5)(5,-5) 
\gcell(1,1)[] \gcell(1,2)[] \gcell(1,3)[] \gcell(1,4)[]
\gcell(2,1)[] \gcell(2,2)[] \gcell(2,3)[] 
\gcell(3,1)[] \gcell(3,2)[] 
\gcell(4,1)[] 
\psdk5
\harrow(1,5)[1] \varrow(2,4)[1] \varrow(3,3)[1] \varrow(4,2)[1] \varrow(5,1)[1]
  \end{pspicture}\qquad
    \begin{pspicture}(0,1.5)(5,-5) 
\gcell(1,1)[] \gcell(1,2)[]
\gcell(2,1)[] \gcell(2,2)[]
\gcell(3,1)[] \gcell(3,2)[] 
\gcell(4,1)[] 
\psdk5
\varrow(4,2)[1] \varrow(5,1)[1]
\rput(-.5,.5){
\rput(3,-1){?} \rput(4,-1){?}
\rput(3,-2){?}
}
\upbrace(3.5,0){3}[$i$]
   \end{pspicture}\qquad
    \begin{pspicture}(0,1.5)(5,-5) 
\gcell(1,1)[] \gcell(1,2)[]
\gcell(2,1)[] \gcell(2,2)[]
\gcell(3,1)[] \gcell(3,2)[] 
\gcell(4,1)[] 
\psdk5
\hzero(3,3)\harrow(4,2)[1] \varrow(5,1)[1]
\rput(-.5,.5){
\rput(3,-1){?} \rput(4,-1){?}
\rput(3,-2){?}
}
\upbrace(3.5,0){3}[$i$]
   \end{pspicture}
    \caption{The elements in $\Dm k\setminus \Dp {k-1}$ containing only bad exceptions.}
    \label{fig:bad_exceptions}
  \end{figure}

  From \eqref{eq:11} and Lemma~\ref{lem:rec} we get the recurrence relation
  \eqref{eq:rec_Tk} for $T_k$, thus completing the proof of
  Theorem~\ref{thm:main2}. 

In order to find a formula for $T_k$ from the above recurrence relation, we
introduce a lattice path model for $T_k$.  We consider the integer lattice
$\mathbb{Z}\times \mathbb{Z}$ in which the unit length is defined to be $\sqrt2$
so that the area of a unit square is $2$. In this lattice the area of the right
triangle with three vertices $(0,0)$, $(1,0)$, and $(0,-1)$ is $1$.

For nonnegative integers $k$ and $j$, let $M[(k,0)\to(0,-j)]$ denote the set of
paths from $(k,0)$ to $(0,-j)$ consisting of west steps $(-1,0)$ and southwest
steps $(-1,-1)$. We define the weight $w(p)$ of $p\in M[(k,0)\to(0,-j)]$ to be
\begin{equation}
  \label{eq:4}
w(p) = (-1)^j q^{A(R)} (1-t^2)^s V,  
\end{equation}
where $A(R)$ is the area of the region $R$ bounded by the $x$-axis, the $y$-axis,
and $p$, $s$ is the number of southwest steps immediately followed by a west
step, and $V=1+t$ if the last step is southwest, and $V=1$ otherwise.

\begin{lem}\label{lem:path}
  For $k\geq0$, we have
  \[
T_k = \sum_{j\geq0} \sum_{p\in M[(k,0)\to(0,-j)]} w(p).
\]
\end{lem}
\begin{proof}
  Let $T'_k$ denote the right hand side of the equation. We will show that
  $T'_k$ satisfies the same recurrence relation in \eqref{eq:rec_Tk}.

  Observe that $T'_k$ is the sum of $w(p)$ for all paths $p$ from $(k,0)$ to a
  point on the $y$-axis consisting of west steps and southwest steps. The weight
  sum of such paths $p$ starting with a west step is $T'_{k-1}$. Suppose now that
  $p$ starts with a southwest step. If $p$ has only southwest steps, then $p$
  must be a path from $(k,0)$ to $(0,-k)$ and $w(p) = (-1)^k q^{k^2}
  (1+t)$. Otherwise we may assume that the first west step of $p$ is the
  $(i+1)$st step for some $1\leq i\leq k-1$. Let $p'$ be the path obtained from
  $p$ by removing the first $i+1$ steps and shifting the remaining path upwards
  by $i$ units. Then $p'$ is a path from $(k-i-1,0)$ to a point on the $y$-axis
  and $w(p) = (-1)^i q^{k^2-(k-i)^2} (1-t^2) w(p')$. Summarizing these, we get
\[
T'_k = T'_{k-1} + (1+t)(-q)^{k^2} + 
(1-t^2) \sum_{i=1}^{k-1} (-1)^i q^{k^2-(k-i)^2} T'_{k-i-1}.
\]
Changing the index $i$ to $k-i$ in the above sum, we obtain that $T_k$ and
$T'_k$ satisfy the same recurrence relation. Since $T_0=T'_0=1$, we have
$T_k=T'_k$.
\end{proof}

\begin{figure}
  \centering
\begin{pspicture}(-3,-1.5)(15,6.5) 
\psgrid[gridcolor=gray](0,0)(12,5) \psset{linewidth=1.5pt}
\psline(0,0)(1,1)(3,1)(5,3)(7,3)(8,4)(9,4)(10,5)(12,5)
\uput[225](0,0){$(0,-j)$}
\uput[135](0,5){$(0,0)$}
\uput[45](12,5){$(k,0)$}
\uput[-45](12,0){$(k,-j)$}
\rput(3,3.5){\Large $R$}
\end{pspicture}%
\begin{pspicture}(0,-1.5)(1,6.5) 
\rput(.5, 2.5){$\Leftrightarrow$}
\end{pspicture}%
\begin{pspicture}(-3,-1.5)(15,6.5) 
\psgrid[gridcolor=gray](0,0)(12,5) \psset{linewidth=1.5pt}
\psline(0,0)(1,1)(3,1)(5,3)(7,3)(8,4)(9,4)(10,5)(12,5)
\psline(1,1)(5,5)
\psline(2,1)(6,5)
\psline(5,3)(7,5)
\psline(6,3)(8,5)
\psline(8,4)(9,5)
\psline(5,5)(10,5)
\psline(4,4)(8,4)
\psline(3,3)(5,3)
\psline(2,2)(4,2)
\pspolygon[linestyle=dashed, linecolor=blue](0,0)(5,5)(12,5)(7,0)
\psdot(0,0)
\psdot(5,5)
\psdot(12,5)
\psdot(7,0)
\uput[225](0,0){$(0,-j)$}
\uput[135](0,5){$(0,0)$}
\uput[45](12,5){$(k,0)$}
\uput[-45](12,0){$(k,-j)$}
\uput[90](5,5){$(j,0)$}
\uput[270](7,0){$(k-j,-j)$}
\end{pspicture}
\caption{An example of $p\in M[(b,k)\to(0,-j)]$. The region $S$ obtained from
  $R$ by removing the right triangle with three vertices $(0,0)$, $(0,-j)$, and
  $(j,0)$ can be identified with the partition $\lambda=(5,4,2,2)\subset
  B(j,k-j)$.}
  \label{fig:M-path}
\end{figure}
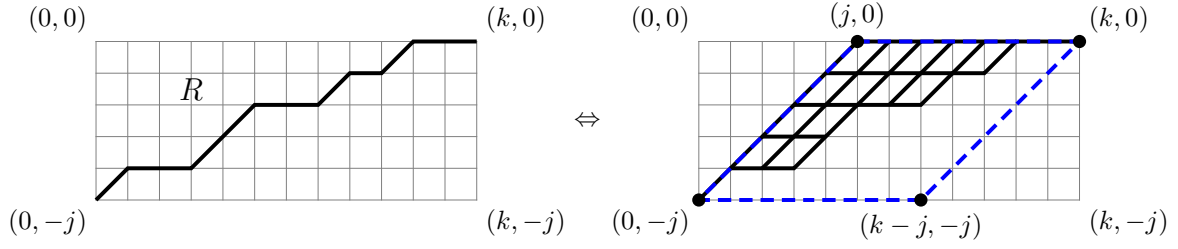
 
Suppose $p\in M[(k,0)\to(0,-j)]$. Then the region $R$ in \eqref{eq:4} contains
the right triangle with three vertices $(0,0)$, $(j,0)$, and $(0,-j)$ whose area
is $j^2$. If we remove this right triangle from $R$, the remaining region $S$
can be identified with a partition $\lambda\subset B(j,k-j)$ as shown in
Figure~\ref{fig:M-path}. Then we have $A(S) = 2|\lambda|$. Moreover, $s$ equals
the number of inner corners of $\lambda$, which is the number $\dist(\lambda)$
of distinct parts, and $V=1+t$ if $\lambda_j=0$, and $V=1$ if
$\lambda_j>0$. Therefore, we have
\begin{equation}
  \label{eq:5}
  w(p) = (-q)^{j^2} q^{2|\lambda|} (1-t^2)^{\dist(\lambda)} V,
\end{equation}
where $V=1+t$ if $\lambda_j=0$, and $V=1$ if $\lambda_j>0$.  Since 
$M[(b,k)\to(0,-j)]=\emptyset$ if $j>k$, we get
\begin{equation}
  \label{eq:6}
T_k = \sum_{j=0}^k (-q)^{j^2} \left( \sum_{\lambda\subset B(j,k-j)}
  q^{2|\lambda|} (1-t^2)^{\dist(\lambda)} + \sum_{\lambda\subset
    B(j-1,k-j)} t q^{2|\lambda|} (1-t^2)^{\dist(\lambda)} \right).
\end{equation}

\begin{lem}\label{lem:dist}
For nonnegative integers $m$ and $n$, we have
\[
\sum_{\lambda\subset B(m,n)} x^{\dist(\lambda)} q^{|\lambda|}  
 =\sum_{i=0}^m q^{\binom{i+1}{2}} \qbinom{n}{i} 
\qbinom{n+m-i}{m-i} (x-1)^i.
\]
\end{lem}
\begin{proof}
  Let $P_n$ denote the set of partitions such that the largest part is at most
  $n$ and every part is nonzero. It is not hard to see that
\[
\sum_{\lambda\in P_n} y^{\ell(\lambda)} x^{\dist(\lambda)} q^{|\lambda|} 
=\prod_{i=1}^n \left( 1+ \frac{yxq^i}{1-yq^i} \right)
= \prod_{i=1}^n \left( 1+ y(x-1)q^i\right) \prod_{j=1}^n \frac{1}{1-yq^j},
\]
where $\ell(\lambda)$ is the number of parts of $\lambda$.  Then by the
$q$-binomial theorem \cite[Exercise~1.2 (vi)]{Gasper2004}, we have
\[
\prod_{i=1}^n \left( 1+ y(x-1)q^i\right)
=\sum_{i=0}^n q^{\binom{i+1}{2}} \qbinom{n}{i} y^i (x-1)^i.
\]
Since the condition $\lambda\subset B(m,n)$ is equivalent to $\lambda\in P_n$
with $\ell(\lambda)\leq m$, we have
\begin{align*}
\sum_{\lambda\subset B(m,n)} x^{\dist(\lambda)} q^{|\lambda|}   &=
[y^{\leq m}] \left( 
\sum_{\lambda\in P_n} y^{\ell(\lambda)} x^{\dist(\lambda)} q^{|\lambda|} \right)\\  
&=[y^{\leq m}] \left( 
\sum_{i=0}^n q^{\binom{i+1}{2}} \qbinom{n}{i} y^i (x-1)^i 
 \prod_{j=1}^n \frac{1}{1-yq^j}
\right)\\
&= \sum_{i=0}^{\min(m,n)} q^{\binom{i+1}{2}} \qbinom{n}{i} (x-1)^i 
\cdot [y^{\leq m-i}] \left( 
\prod_{j=1}^n \frac{1}{1-yq^j}
\right),
\end{align*}
where $[y^{\leq m}] f(y)$ means the sum of the coefficients of $y^j$ in $f(y)$
for $j\leq m$. Note that it is no harm to replace $\min(m,n)$ with $m$ in the
last sum of the above equation. Since
\[
[y^{\leq m-i}] \left( \prod_{j=1}^n \frac{1}{1-yq^j} \right)
= \sum_{\lambda\subset B(m-i, n)} q^{|\lambda|}
=\qbinom{n+m-i}{m-i},
\] 
we are done.
\end{proof}

Now we can complete the proof of Corollary~\ref{cor:Tk}.

\begin{proof}[Proof of Corollary~\ref{cor:Tk}]
  Applying Lemma~\ref{lem:dist} to \eqref{eq:6}, we obtain that $T_k$ is equal
  to
\begin{align*}
&\sum_{j=0}^k (-q)^{j^2} \left(
\sum_{i=0}^j q^{i^2+i} \Qbinom{k-j}{i}{q^2} \Qbinom{k-i}{j-i}{q^2}(-t^2)^i
+ \sum_{i=0}^{j-1} t q^{i^2+i} 
\Qbinom{k-j}{i}{q^2} \Qbinom{k-i-1}{j-i-1}{q^2}(-t^2)^i
\right),
\end{align*}
which gives the desired formula.
\end{proof}

\section{Self-conjugate overpartitions}
\label{sec:self-conj-overp}

In this section we will express the sum $T_k(t,q)$ in the previous section using
overpartitions. Overpartitions were first introduced by Corteel and Lovejoy
\cite{Corteel2004}. We define overpartitions in a slightly different way, but it
should be clear that the two definitions are equivalent. 

\begin{defn}
  An \emph{overpartition} is a partition in which each inner corner may be
  marked. For an overpartition $\lambda$, we define the \emph{conjugate} of
  $\lambda$ in the natural way: the partition is transposed and the cell $(i,j)$
  is marked if and only if the cell $(j,i)$ is marked in $\lambda$, see
  Figure~\ref{fig:overpartition}.  A \emph{self-conjugate overpartition} is an
  overpartition whose conjugate is equal to itself. We denote by $\sop(k)$ the
  set of self-conjugate overpartitions whose underlying partitions are contained
  in $B(k,k)$. A \emph{diagonal cell} is the $(i,i)$-cell for some $i$.  For an
  overpartition $\lambda$, the number of diagonal cells is denoted by
  $\diag(\lambda)$, and the number of marked cells is denoted by
  $\mk(\lambda)$. The \emph{main diagonal} is the infinite set of $(i,i)$-cells
  (not necessarily contained in $\lambda$) for all $i$.
\end{defn}

 \begin{figure}
   \centering
\begin{pspicture}(0,0)(5,-5)
   \psline(5,0)(0,0)(0,-4)
   \cell(1,1)[] \cell(1,2)[]   \cell(1,3)[]   \cell(1,4)[]   \cell(1,5)[]
   \cell(2,1)[] \cell(2,2)[]   \cell(2,3)[]   \cell(2,4)[]  
   \cell(3,1)[] \cell(3,2)[]   
   \cell(4,1)[] 
   \mcell(2,4) \mcell(4,1)
\end{pspicture}%
\begin{pspicture}(0,0)(6,-5) 
\rput(3, -3){$\Leftrightarrow$}
\rput(3,-2){conjugate}
\end{pspicture}%
 \begin{pspicture}(0,0)(5,-5)
   \psline(4,0)(0,0)(0,-5)
   \cell(1,1)[] \cell(1,2)[]   \cell(1,3)[]   \cell(1,4)[]   
   \cell(2,1)[] \cell(2,2)[]   \cell(2,3)[]   
   \cell(3,1)[] \cell(3,2)[]   
   \cell(4,1)[] \cell(4,2)[]
   \cell(5,1)[] 
   \mcell(1,4) \mcell(4,2)
\end{pspicture}
   \caption{An overpartition and its conjugate}
   \label{fig:overpartition}
 \end{figure}
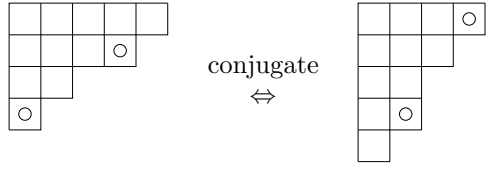

\begin{figure}
   \centering
\begin{pspicture}(0,1.5)(5,-8) 
\psline(3,0)(0,0)(0,-4)
   \cell(1,1)[] \cell(1,2)[]   \cell(1,3)[]  
   \cell(2,1)[] \cell(2,2)[] 
   \cell(3,1)[] \cell(3,2)[]   
   \cell(4,1)[] \mcell(3,2) \mcell(1,3)
\end{pspicture}
\begin{pspicture}(-3,1.5)(8,-8) 
\rput(-3,-2){$\Rightarrow$}
\psframe(0,0)(5,-5) 
\rput(5,0){\psline(3,0)(0,0)
   \cell(1,1)[] \cell(1,2)[]   \cell(1,3)[]  
   \cell(2,1)[] \cell(2,2)[] 
   \cell(3,1)[] \cell(3,2)[]   
   \cell(4,1)[] \mcell(3,2) \mcell(1,3)
}
\rput(0,-5){\psline(0,0)(0,-3)
   \cell(1,1)[] \cell(1,2)[]   \cell(1,3)[]  \cell(1,4)[]  
   \cell(2,1)[] \cell(2,2)[]  \cell(2,3)[]  
   \cell(3,1)[] \mcell(2,3) \mcell(3,1)
}
 \upbrace(2.5,0){5}[$j$]
 \leftbrace(0,-2.5){5}[$j$]
\end{pspicture}
\caption{The construction of $\nu\in\sop(k)$ from an overpartition $\lambda$
  whose underlying partition is contained in $B(j,k-j)$.}
   \label{fig:self-conjugate}
 \end{figure}
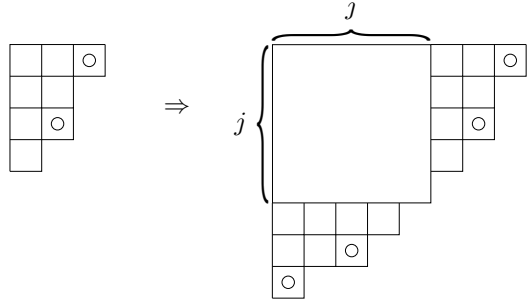

 Recall that by Lemma~\ref{lem:path} and \eqref{eq:5} we have
 \begin{equation}
   \label{eq:7}
T_k(t,q)=\sum_{j=0}^k \sum_{\lambda \subset B(j,k-j)} (-q)^{j^2} q^{2|\lambda|}
(1-t^2)^{\dist(\lambda)} V,
\end{equation}
where $\dist(\lambda)$ is the number of distinct parts of $\lambda$, and $V=1+t$
if $\lambda_j=0$, and $V=1$ if $\lambda_j>0$. Since $\dist(\lambda)$ is equal to
the number of inner corners of $\lambda$, the factor $(1-t^2)^{\dist(\lambda)}$
in \eqref{eq:7} can be understood as marking each inner corner or not. Thus
\eqref{eq:7} can be rewritten as
\begin{equation}
  \label{eq:8}
T_k(t,q)=\sum_{j=0}^k \sum_{\lambda} (-1)^{j+\mk(\lambda)} t^{2\mk(\lambda)} 
q^{j^2} q^{2|\lambda|}V,
\end{equation}
where the latter sum is over all overpartitions $\lambda$ whose underlying
partitions are contained in $B(j,k-j)$. For such an overpartition $\lambda$, we
construct $\nu\in \sop(k)$ which is obtained from the box $B(j,j)$ by attaching
$\lambda$ to the right of the box and its conjugate to the bottom of the box as
shown in Figure~\ref{fig:self-conjugate}. Then $\nu$ always has even number of
marked cells and 
\[
(-1)^{j+\mk(\lambda)} t^{2\mk(\lambda)} 
q^{j^2} q^{2|\lambda|} = 
(-1)^{\diag(\nu)+\frac{\mk(\nu)}2} t^{\mk(\nu)} q^{|\nu|}.
\]
On the other hand, in \eqref{eq:8} $V=1+t$ if $\lambda_j =0$, and $V=1$ if
$\lambda_j>0$, equivalently, $V=1+t$ if $\nu$ has an inner corner on the main
diagonal, and $V=1$ otherwise.  Considering $V=1+ t$ as marking the diagonal
inner corner or not, we can express $T_k(t,q)$ as follows.

\begin{thm}\label{thm:sop}
We have
\[
T_k(t,q)=\sum_{\nu \in \sop(k)} (-1)^{\diag(\nu)+\flr{\frac{\mk(\nu)}2}} 
 t^{\mk(\nu)} q^{|\nu|}.
\]
\end{thm}

We close this section by finding a functional equation for $T_k(t,q)$ which will
serve as a recurrence relation in the next section.

\begin{cor}\label{cor:rec_T}
  For $k\geq1$, we have
\[
(1-tq) T_k(tq,q) = T_k(t,q) + t^2 q^{2k+1} T_{k-1}(t,q).
\]
\end{cor}
\begin{proof}
  For $\nu\in \sop(k)$, let
  $\omega(\nu)=(-1)^{\diag(\nu)+\flr{\frac{\mk(\nu)}2}} t^{\mk(\nu)} q^{|\nu|}$. Then
\[
T_k(t,q)=\sum_{\nu \in \sop(k)} \omega(\nu).
\]

We can think of $\omega(\nu)$ as the product of the weight of the cells and
marks in $\nu$, which are defined as follows:
\begin{enumerate}
\item every non-diagonal cell has weight $q$,
\item every diagonal cell has weight $-q$,
\item every mark above the main diagonal has weight $-t$, and
\item every mark below or on the main diagonal has weight $t$.
\end{enumerate}

In order to express the left hand side of the equation we define $\sop'(k)$ to
be the set of $\nu\in\sop(k)$ in which the unique corner on the main diagonal
may have a special mark. Note that the corner of the main diagonal can be an
inner corner or an outer corner depending on $\nu$, and if it is an inner
corner, then this corner may have two marks, one is non-special and the other is
special.  For $\nu\in \sop'(k)$, we define $\omega'(\nu)$ to be the product of
weights of the cells and marks, which are defined as follows:
\begin{enumerate}
\item every non-diagonal cell has weight $q$,
\item every diagonal cell has weight $-q$,
\item every mark above the main diagonal has weight $-tq$,
\item every mark below or on the main diagonal has weight $tq$, and
\item if there is a special mark, it has weight $-tq$.
\end{enumerate}
It is easy to see that
\[
(1-tq)T_k(tq,q)=\sum_{\nu \in \sop'(k)} \omega'(\nu).
\]

Let $X$ be the set of $\nu \in \sop'(k)$ which has an inner corner on the main
diagonal with only one mark. For $\nu\in X$, we define $\nu'$ to be the element
in $X$ that is obtained by switching the mark in the inner corner on the main
diagonal to special one or non-special one. It is clear that
$\omega'(\nu')=-\omega'(\nu)$. Thus the sum of $\omega'(\nu)$ for all $\nu\in X$
is zero and we get
\[
(1-tq)T_k(tq,q)=\sum_{\nu \in \sop'(k)\setminus X} \omega'(\nu).
\]

Now suppose $\nu\in \sop'(k)\setminus X$. For each mark above (resp.~below) the
main diagonal, if it is in Row $i$ (resp.~Column $i$), delete the mark and add a
cell in Row $i+1$ (resp.~Column $i+1$) and mark the new cell. If there is a
special mark in the outer corner on the diagonal, then add a cell to $\nu$ to
fill this outer corner and change the special mark to a non-special mark, see
Figure~\ref{fig:outer}. If there are one non-special mark and one special mark
in the inner corner on the main diagonal, which is in Row $i$ and Column $i$,
then delete the two marks, add one cell to Row $i+1$ and one cell to Column
$i+1$, and mark the two new cells, see Figure~\ref{fig:inner}. Let $\mu$ be the
resulting overpartition. From the construction it is clear that $\omega'(\nu) =
\omega(\mu)$. Also, it is not hard to see that $\mu$ is an element in $\sop(k)$
or an element in $\sop(k+1)$. Moreover, if $\mu\in\sop(k+1)$, the $(1,k+1)$-cell
and the $(k+1,1)$-cell of $\mu$ are marked inner corners, and the overpartition
$\mu'$ obtained from $\mu$ by deleting Row $1$ and Column $1$ satisfies $\mu'\in
\sop(k-1)$ and $\omega(\mu)=t^2 q^{2k+1} \omega(\mu')$. Note that the sign does
not change because $\mu'$ has one less diagonal cells and two less marks than
$\mu$.  Thus we have
\[
\sum_{\nu \in \sop'(k)\setminus X} \omega'(\nu) = 
\sum_{\nu \in \sop(k)} \omega(\nu)
+ t^2q^{2k+1} \sum_{\nu \in \sop(k-1)} \omega(\nu),
\]
which finishes the proof.
\end{proof}

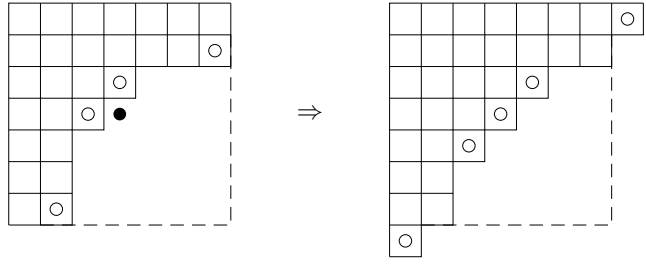
\begin{figure}
  \centering
\begin{pspicture}(0,0)(8,-8)
\cell(1,1)[]\cell(1,2)[]\cell(1,3)[]\cell(1,4)[]\cell(1,5)[]\cell(1,6)[]\cell(1,7)[]
\cell(2,1)[]\cell(2,2)[]\cell(2,3)[]\cell(2,4)[]\cell(2,5)[]\cell(2,6)[]\cell(2,7)[]  
\cell(3,1)[]\cell(3,2)[]\cell(3,3)[]\cell(3,4)[]   
\cell(4,1)[]\cell(4,2)[]\cell(4,3)[] 
\cell(5,1)[]\cell(5,2)[]
\cell(6,1)[]\cell(6,2)[]
\cell(7,1)[]\cell(7,2)[]
\mcell(2,7) \mcell(7,2) \mcell(3,4) \mcell(4,3) \smcell(4,4)
\psline(0,-7)(0,0)(7,0) \psline[linestyle=dashed](1,-7)(7,-7)(7,-1)
\end{pspicture}%
\begin{pspicture}(0,0)(4,-8)
\rput(1.5,-3.5){$\Rightarrow$}
\end{pspicture}%
\begin{pspicture}(0,0)(8,-8)
\cell(1,1)[]\cell(1,2)[]\cell(1,3)[]\cell(1,4)[]\cell(1,5)[]\cell(1,6)[]\cell(1,7)[]\cell(1,8)[]
\cell(2,1)[]\cell(2,2)[]\cell(2,3)[]\cell(2,4)[]\cell(2,5)[]\cell(2,6)[]\cell(2,7)[]    
\cell(3,1)[]\cell(3,2)[]\cell(3,3)[]\cell(3,4)[]\cell(3,5)[]  
\cell(4,1)[]\cell(4,2)[]\cell(4,3)[]\cell(4,4)[] 
\cell(5,1)[]\cell(5,2)[]\cell(5,3)[]
\cell(6,1)[]\cell(6,2)[]
\cell(7,1)[]\cell(7,2)[]
\cell(8,1)[]
\mcell(1,8) \mcell(8,1) \mcell(3,5) \mcell(5,3) \mcell(4,4)
\psline(0,-8)(0,0)(8,0) \psline[linestyle=dashed](1,-7)(7,-7)(7,-1)
\end{pspicture}%
\caption{Moving the marks in $\nu\in\sop'(k)$ when there is a special mark in
   the outer corner on the main diagonal.}
  \label{fig:outer}
\end{figure}

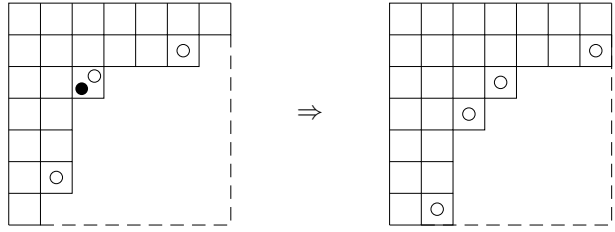
\begin{figure}
  \centering
\begin{pspicture}(0,0)(8,-7)
\cell(1,1)[]\cell(1,2)[]\cell(1,3)[]\cell(1,4)[]\cell(1,5)[]\cell(1,6)[]\cell(1,7)[]
\cell(2,1)[]\cell(2,2)[]\cell(2,3)[]\cell(2,4)[]\cell(2,5)[]\cell(2,6)[]  
\cell(3,1)[]\cell(3,2)[]\cell(3,3)[]
\cell(4,1)[]\cell(4,2)[]
\cell(5,1)[]\cell(5,2)[]
\cell(6,1)[]\cell(6,2)[] 
\cell(7,1)[]
\mcell(2,6) \mcell(6,2) 
\rput(.2,.2){\mcell(3,3)}
\rput(-.2,-.2){\smcell(3,3)}
\psline(0,-7)(0,0)(7,0) \psline[linestyle=dashed](1,-7)(7,-7)(7,-1)
\end{pspicture}%
\begin{pspicture}(0,0)(4,-7)
\rput(1.5,-3.5){$\Rightarrow$}
\end{pspicture}%
\begin{pspicture}(0,0)(8,-7)
\cell(1,1)[]\cell(1,2)[]\cell(1,3)[]\cell(1,4)[]\cell(1,5)[]\cell(1,6)[]\cell(1,7)[]
\cell(2,1)[]\cell(2,2)[]\cell(2,3)[]\cell(2,4)[]\cell(2,5)[]\cell(2,6)[]\cell(2,7)[]
\cell(3,1)[]\cell(3,2)[]\cell(3,3)[]\cell(3,4)[]
\cell(4,1)[]\cell(4,2)[]\cell(4,3)[]
\cell(5,1)[]\cell(5,2)[]
\cell(6,1)[]\cell(6,2)[]
\cell(7,1)[]\cell(7,2)[]
\mcell(2,7) \mcell(7,2) \mcell(3,4) \mcell(4,3)
\psline(0,-7)(0,0)(7,0) \psline[linestyle=dashed](1,-7)(7,-7)(7,-1)
\end{pspicture}%
 \caption{Moving the marks in $\nu\in\sop'(k)$ when there is a special mark in
   the inner corner on the main diagonal.}
  \label{fig:inner}
\end{figure}

\section{Another formula for $T_k(\pm q^r,q)$}
\label{sec:another_formula}

In this section we will find another formula for $T_k(t,q)$ when $t=\pm q^r$ for
any integer $r$.  To this end we need to divide the cases when $r\geq0$ and
$r\leq 0$.  For a sign $\epsilon\in \{+,-\}$, and nonnegative integers $b$ and
$k$, we define
\[
\aep(b,k) = T_k(\epsilon q^b, q), \qquad \bep(b,k) = T_k(\epsilon q^{-b}, q).
\]
Note that for $b\geq0$, we have
\begin{equation}
  \label{eq:initial k=0}
\aep(b,0)  = \bep(b,0)  =1.
\end{equation}
Recall that from the recurrence \eqref{eq:rec_Tk} of $T_k(t,q)$, we immediately
get $T_k(-1,q)=1$ and $T_k(1,q)=\sum_{i=-k}^k (-q)^{i^2}$.  Thus we have
\begin{equation}
  \label{eq:(0,k)-}
\alpha_-(0,k) = \beta_-(0,k) = T_k(-1,q) = 1,
\end{equation}
\begin{equation}
  \label{eq:(0,k)+}
\alpha_+(0,k) = \beta_+(0,k) = T_k(1,q) = \sum_{i=-k}^k (-q)^{i^2}.
\end{equation}
  
Substituting $t=\epsilon q^{b-1}$ in Corollary~\ref{cor:rec_T}, we obtain
\[
(1-\epsilon q^b) T_k(\epsilon q^b,q) 
= T_k(\epsilon q^{b-1},q) +q^{2k+2b-1} T_{k-1}(\epsilon q^{b-1},q).
\]
If $b\geq1$, we can divide the both sides of the above equation by $1-\epsilon
q^b$  to get the following lemma.

\begin{lem}\label{lem:aep}
For integers $b,k\geq1$, we have
\[
\aep(b,k)= \frac{1}{1-\epsilon q^b} \aep(b-1,k)
+\frac{q^{2k+2b-1}}{1-\epsilon q^b} \aep(b-1,k-1).
\]  
\end{lem}

Substituting $t=\epsilon q^{-b}$ in Corollary~\ref{cor:rec_T}, we obtain
\[
(1-\epsilon q^{1-b}) T_k(\epsilon q^{1-b},q) 
= T_k(\epsilon q^{-b},q) +q^{2k-2b+1} T_{k-1}(\epsilon q^{-b},q),
\]
which implies the following lemma.

\begin{lem}\label{lem:bep}
For integers $b,k\geq1$, we have
\[
\bep(b,k)= (1-\epsilon q^{1-b}) \bep(b-1,k) - q^{2k-2b+1} \bep(b,k-1).
\]  
\end{lem}

Now we have recurrence relations and initial conditions for $\aep(b,k)$ and
$\bep(b,k)$. Thus we can use the idea in Section~\ref{sec:self-conj-overp} to
compute $\aep(b,k)$ and $\bep(b,k)$.  As we did in
Section~\ref{sec:self-conj-overp} we define the unit length in the lattice
$\mathbb{Z}\times \mathbb{Z}$ to be $\sqrt2$.

\subsection{Formula for $T_k(\pm q^r,q)$ when $r\geq0$}

Suppose $m$ and $n$ are nonnegative integers with $m=0$ or $n=0$. We define
$L[(b,k)\to(m,n)]$ to be the set of lattice paths from $(b,k)$ to $(m,n)$
consisting of west steps $(-1,0)$ and southwest steps $(-1,-1)$ without any west
steps on the $x$-axis. The condition that there is no west step on the $x$-axis
guarantees that the lattice path ends when it first touches the $x$-axis or the
$y$-axis.

For $p\in L[(b,k)\to(m,n)]$ we define the weight $w(p)$ by
\begin{equation}
  \label{eq:9}
w(p) = q^{A(R)} \prod_{i=m+1}^b \frac1{1-\epsilon q^i} \prod_{i=n+1}^k q^{2i},  
\end{equation}
where $A(R)$ is the area of the upper region $R$ of the rectangle with four
vertices $(0,0)$, $(b,0)$, $(0,k)$, and $(b,k)$ divided by the path $p$.

\begin{lem}\label{lem:alpha_wt}
For $b,k\geq0$, we have
  \[
\aep(b,k) = \sum_{\substack{m,n\geq0\\ mn=0}} \aep(m,n) \sum_{p\in L[(b,k)\to(m,n)]} w(p).
\]
\end{lem}
\begin{proof}
  Let $F(b,k)$ denote the right hand side and let $f_{m,n}(b,k)$ denote the
  latter sum there. Using a similar argument as in the proof of
  Lemma~\ref{lem:path}, one can easily check that for $b,k\geq1$,
\[
f_{m,n}(b,k)= \frac{1}{1-\epsilon q^b} f_{m,n}(b-1,k)
+\frac{q^{2k+2b-1}}{1-\epsilon q^b} f_{m,n}(b-1,k-1).
\]  
Thus $F(b,k)$ and $\aep(b,k)$ satisfy the same recurrence relation. Since
$F(b,k)=\aep(b,k)$ when $b=0$ or $k=0$, we get $F(b,k)=\aep(b,k)$ for all
$b,k\geq0$. 
\end{proof}

Since $\aep(m,0)=1$, the formula in the previous lemma can be written as
\begin{equation}
  \label{eq:12}
\aep(b,k) = \sum_{n\geq1} \aep(0,n)\sum_{p\in L[(b,k)\to(0,n)]} w(p) 
+\sum_{m\geq0} \sum_{p\in L[(b,k)\to(m,0)]} w(p). 
\end{equation}

Now we compute the weight sums in \eqref{eq:12}.

\begin{lem}\label{lem:(0,n)}
For $b,k\geq0$ and $n\geq1$, we have
\[
\sum_{p\in L[(b,k)\to(0,n)]} w(p) =  \frac{q^{(k-n)(2k+1)}}{(\epsilon q;q)_b} \Qbinom{b}{k-n}{q^2}. 
\]
\end{lem}
\begin{proof}
  Let $p\in L[(b,k)\to(0,n)]$. From the definition of $w(p)$ in \eqref{eq:9},
  we have
\[
w(p) =  \frac{q^{k(k+1)-n(n+1)}}{(\epsilon q;q)_b} \cdot q^{A(R)}.
\]
Since $p$ consists of west steps and southwest steps, the region contains the
right triangle with three vertices $(0,n)$, $(0,k)$, and $(k-n,k)$, whose area
is $(k-n)^2$, see Figure~\ref{fig:L-path}. Let $S$ be the region obtained from
$R$ by removing this right triangle. Then $S$ is contained in the quadrilateral
with four vertices $(0,n)$, $(k-n,k)$, $(b,k)$, and $(b-k+n,n)$. Again by the
fact that $p$ consists of west steps and southwest steps, one can identify $S$
with a partition $\lambda$ contained in $B(k-n,b-k+n)$. In this identification
we have $A(S) = 2|\lambda|$. Thus, we get
\begin{align*}
\sum_{p\in L[(b,k)\to(0,n)]} w(p)
&=\frac{q^{k(k+1) - n(n+1)}}{(\epsilon q;q)_b}\cdot  q^{(k-n)^2} 
\sum_{\lambda\subset B(k-n,b-k+n)}  q^{2|\lambda|}\\
&=\frac{q^{(k-n)(2k+1)}}{(\epsilon q;q)_b} \Qbinom{b}{k-n}{q^2}. 
\end{align*}
\end{proof}

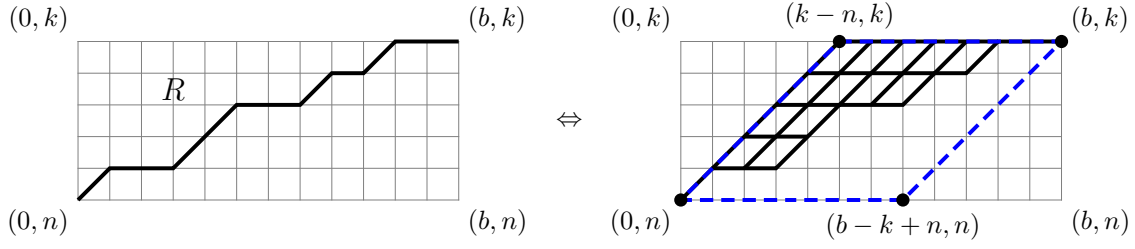
\begin{figure}
  \centering
\begin{pspicture}(-3,-1.5)(15,6.5) 
\psgrid[gridcolor=gray](0,0)(12,5) \psset{linewidth=1.5pt}
\psline(0,0)(1,1)(3,1)(5,3)(7,3)(8,4)(9,4)(10,5)(12,5)
\uput[225](0,0){$(0,n)$}
\uput[135](0,5){$(0,k)$}
\uput[45](12,5){$(b,k)$}
\uput[-45](12,0){$(b,n)$}
\rput(3,3.5){\Large $R$}
\end{pspicture}%
\begin{pspicture}(0,-1.5)(1,6.5) 
\rput(.5, 2.5){$\Leftrightarrow$}
\end{pspicture}%
\begin{pspicture}(-3,-1.5)(15,6.5) 
\psgrid[gridcolor=gray](0,0)(12,5) \psset{linewidth=1.5pt}
\psline(0,0)(1,1)(3,1)(5,3)(7,3)(8,4)(9,4)(10,5)(12,5)
\psline(1,1)(5,5)
\psline(2,1)(6,5)
\psline(5,3)(7,5)
\psline(6,3)(8,5)
\psline(8,4)(9,5)
\psline(5,5)(10,5)
\psline(4,4)(8,4)
\psline(3,3)(5,3)
\psline(2,2)(4,2)
\pspolygon[linestyle=dashed, linecolor=blue](0,0)(5,5)(12,5)(7,0)
\psdot(0,0)
\psdot(5,5)
\psdot(12,5)
\psdot(7,0)
\uput[135](0,5){$(0,k)$}
\uput[225](0,0){$(0,n)$}
\uput[90](5,5){$(k-n,k)$}
\uput[45](12,5){$(b,k)$}
\uput[270](7,0){$(b-k+n,n)$}
\uput[-45](12,0){$(b,n)$}
\end{pspicture}
  \caption{An example of $p\in L[(b,k)\to(0,n)]$.}
  \label{fig:L-path}
\end{figure}

\begin{lem}\label{lem:(m,0)}
For $b\geq0$, $k\geq1$ and $m\geq0$, we have
\[
\sum_{p\in L[(b,k)\to(m,0)]} w(p) =  
\frac{(\epsilon q;q)_m}{(\epsilon q;q)_b} q^{k(2k+2m+1)} \Qbinom{b-m-1}{k-1}{q^2}.
\]
\end{lem}
\begin{proof}
  This is similar to the proof of the previous lemma.  The only difference is
  that since the last step of $p$ is always a southwest step, $p$ visits
  $(m+1,1)$ right before its end point. Then the same argument works, so we omit
  the details.
\end{proof}

Finally, we obtain a formula for $\aep(b,k)$.

\begin{thm}\label{thm:aep}
For $b\geq0$ and $k\geq1$, we have
\[
\aep(b,k) = 
\sum_{i=0}^{k-1} \frac{q^{i(2k+1)}}{(\epsilon q;q)_b} \Qbinom{b}{i}{q^2} \aep(0,k-i)
+\sum_{i=0}^{b-1}\frac{(\epsilon q;q)_i}{(\epsilon q;q)_b} q^{k(2k+2i+1)} 
\Qbinom{b-i-1}{k-1}{q^2}.
\]
\end{thm}
\begin{proof}
By \eqref{eq:12} and Lemmas~\ref{lem:(0,n)} and \ref{lem:(m,0)}, we have
\[
\aep(b,k) = 
\sum_{n\geq1} \frac{q^{(k-n)(2k+1)}}{(\epsilon q;q)_b} \Qbinom{b}{k-n}{q^2} \aep(0,n)
+\sum_{m\geq0} \frac{(\epsilon q;q)_m}{(\epsilon q;q)_b} q^{k(2k+2m+1)} 
\Qbinom{b-m-1}{k-1}{q^2}.
\]
In the first sum the summand is zero unless $k-n\geq0$, and in the second sum
the summand is zero unless $m\leq b-1$. Replacing $k-n$ with $i$ in the first sum
and $m$ with $i$ in the second sum we get the desired formula.
\end{proof}

By Theorem~\ref{thm:aep} with $\epsilon=+$ and \eqref{eq:(0,k)+}, we get a
formula for $T_k(q^b,q)$.
\begin{cor}\label{cor:(b,k)+}
For $b\geq0$ and $k\geq1$, we have
\[
T_k(q^b,q) = 
\sum_{i=0}^{k-1} \frac{q^{i(2k+1)}}{(q;q)_b} \Qbinom{b}{i}{q^2} 
\sum_{j=-(k-i)}^{k-i} (-q)^{j^2}
+\sum_{i=0}^{b-1}\frac{(q;q)_i}{(q;q)_b} q^{k(2k+2i+1)} 
\Qbinom{b-i-1}{k-1}{q^2}.
\]
\end{cor}

If $b=1$ in Corollary~\ref{cor:(b,k)+}, we have that for $k\geq1$,
\[
T_k(q,q) = 
\frac{1}{1-q} \sum_{i=-k}^{k} (-q)^{i^2}
+ \frac{q^{2k+1}}{1-q} \sum_{i=-(k-1)}^{k-1} (-q)^{i^2},
\]
which together with Theorem~\ref{thm:main2} implies \eqref{eq:tan}.

By Theorem~\ref{thm:aep} with $\epsilon=-$ and \eqref{eq:(0,k)-}, we get a
formula for $T_k(-q^b,q)$.

\begin{cor}\label{cor:(b,k)-}
For $b\geq0$ and $k\geq1$, we have
\[
T_k(-q^b,q) = 
\sum_{i=0}^{k-1} \frac{q^{i(2k+1)}}{(-q;q)_b} \Qbinom{b}{i}{q^2} 
+\sum_{i=0}^{b-1}\frac{(-q;q)_i}{(-q;q)_b} q^{k(2k+2i+1)} 
\Qbinom{b-i-1}{k-1}{q^2}.
\]
\end{cor}

If $b=1$ in Corollary~\ref{cor:(b,k)-}, we have that for $k\geq1$,
\begin{equation}
  \label{eq:10}
T_k(-q,q) = \frac{1+q^{2k+1}}{1+q}.  
\end{equation}
Note that the above identity is also true for $k=0$.  This gives the following
formula for $E_n(-q,q)$.

\begin{prop}\label{thm:-q}
We have
 \[ E_n(-q,q) = \frac{1}{(1-q)^{2n}} \sum_{k=0}^n \left( \binom{2n}{n-k} -
    \binom{2n}{n-k-1} \right) (-1)^k\frac{q^{k^2}+q^{(k+1)^2}}{1+q}.
\]
\end{prop}
\begin{proof}
By Theorem~\ref{thm:main2}, 
 \[
 E_n(-q,q) = \frac{1}{(1-q)^{2n}} \sum_{k=0}^n \left( \binom{2n}{n-k} -
  \binom{2n}{n-k-1} \right) (-q)^kq^{k(k+1)} T_k(-q^{-1},q^{-1}).
\]
By \eqref{eq:10} we get
\[
(-q)^kq^{k(k+1)} T_k(-q^{-1},q^{-1})  
= (-q)^{k^2+2k} \frac{1+q^{-2k-1}}{1+q^{-1}}
= (-1)^k\frac{q^{k^2}+q^{(k+1)^2}}{1+q},
\]
which finishes the proof.
\end{proof}

\subsection{Formula for $T_k(\pm q^r,q)$ when $r\leq0$}

Suppose $m$ and $n$ are nonnegative integers with $m=0$ or $n=0$. We define
$L'[(b,k)\to(m,n)]$ to be the set of lattice paths from $(b,k)$ to $(m,n)$
consisting of west steps $(-1,0)$ and south steps $(0,-1)$ without west steps on
the $x$-axis nor south steps on the $y$-axis.  For $p\in L'[(b,k)\to(m,n)]$ we
define the weight $w(p)$ by
\begin{equation}
  \label{eq:9'}
w(p) = q^{-A(R)} \prod_{i=m+1}^{b}(1-\epsilon q^{1-i})
 \prod_{i=n+1}^k(-q^{2i+1}),
\end{equation}
where $A(R)$ is the area of the upper region $R$ of the rectangle with four
vertices $(0,0)$, $(b,0)$, $(0,k)$, and $(b,k)$ divided by the path $p$.

\begin{lem}
For $b,k\geq0$, we have
  \[
\bep(b,k) = \sum_{\substack{m,n\geq0\\ mn=0}} \bep(m,n) \sum_{p\in L'[(b,k)\to(m,n)]} w(p).
\]
\end{lem}
\begin{proof}
  Since this can be done similarly as in the proof of Lemma~\ref{lem:alpha_wt},
  we omit the proof.
\end{proof}

Notice that $L'[(b,k)\to(0,0)]=\emptyset$ unless $(b,k)=(0,0)$.  Since
$\bep(m,0)=1$, the formula in the previous lemma can be written as follows: if
$(b,k)\ne (0,0)$, we have
\begin{equation}
  \label{eq:12'}
\bep(b,k) = \sum_{n\geq1} \bep(0,n)\sum_{p\in L'[(b,k)\to(0,n)]} w(p) 
+\sum_{m\geq1} \sum_{p\in L'[(b,k)\to(m,0)]} w(p). 
\end{equation}

\begin{lem}\label{lem:(0,n)'}
For $b,k\geq0$ and $n\geq1$ with $(b,k)\ne(0,0)$, we have
\[
\sum_{p\in L'[(b,k)\to(0,n)]} w(p) =  
(\epsilon q^{1-b};q)_b (-q)^{(k-n)(k+n-2b+2)}
\Qbinom{b+k-n-1}{k-n}{q^2}.  
\]
\end{lem}
\begin{proof}
  Let $p\in L'[(b,k)\to(0,n)]$. From the definition of $w(p)$ in \eqref{eq:9'},
  we have
\[
w(p) =  q^{-A(R)} (\epsilon q^{1-b}; q)_b (-1)^{k-n} q^{(k+1)^2-(n+1)^2}.
\]
Note that $R$ is contained in the rectangle with four vertices $(0,n)$, $(0,k)$,
$(b,n)$, and $(b,k)$, see Figure~\ref{fig:L'-path}. Let $R'$ be the region of
this rectangle minus $R$. Then $-A(R) = - 2b(k-n) + A(R')$. Since the last step
of $p$ is a west step, $R'$ can be identified with a partition $\lambda\subset
B(k-n,b-1)$, which is rotated by an angle of $180^\circ$, and $A(R')=
2|\lambda|$. Therefore,
\begin{align*}
\sum_{p\in L'[(b,k)\to(0,n)]} w(p) 
&=  (\epsilon q^{1-b}; q)_b (-1)^{k-n} q^{(k-n)(k+n+2) - 2b(k-n)}
\sum_{\lambda\subset B(k-n,b-1)} q^{2|\lambda|}\\
&= (\epsilon q^{1-b};q)_b (-q)^{(k-n)(k+n-2b+2)} \Qbinom{b+k-n-1}{k-n}{q^2}.
\end{align*}
\end{proof}

\begin{figure}
  \centering
\begin{pspicture}(-3,-1.5)(15,6.5) 
\psgrid[gridcolor=gray](0,0)(12,5) \psset{linewidth=1.5pt}
\psline(0,0)(3,0)(3,1)(5,1)(5,3)(8,3)(8,4)(12,4)(12,5)
\uput[225](0,0){$(0,n)$}
\uput[135](0,5){$(0,k)$}
\uput[45](12,5){$(b,k)$}
\uput[-45](12,0){$(b,n)$}
\rput(3,3.5){\Large $R$}
\rput(9,1.5){\Large $R'$}
\end{pspicture}%
\begin{pspicture}(0,-1.5)(1,6.5) 
\rput(.5, 2.5){$\Leftrightarrow$}
\end{pspicture}%
\begin{pspicture}(-3,-1.5)(15,6.5) 
\psgrid[gridcolor=gray](0,0)(12,5) \psset{linewidth=1.5pt}
\psline(3,0)(3,1)(5,1)(5,3)(8,3)(8,4)(12,4)
\pspolygon[linestyle=dashed, linecolor=blue](1,0)(1,5)(12,5)(12,0)
\psline(3,0)(12,0)
\psline(5,1)(12,1)
\psline(5,2)(12,2)
\psline(8,3)(12,3)
\psline(4,0)(4,1)
\psline(5,0)(5,1)
\psline(6,0)(6,3)
\psline(7,0)(7,3)
\psline(8,0)(8,3)
\psline(9,0)(9,4)
\psline(10,0)(10,4)
\psline(11,0)(11,4)
\psline(12,0)(12,4)
\psdot(1,0)
\psdot(1,5)
\psdot(12,5)
\psdot(12,0)
\uput[135](1,5){$(1,k)$}
\uput[225](1,0){$(1,n)$}
\uput[45](12,5){$(b,k)$}
\uput[-45](12,0){$(b,n)$}
\end{pspicture}
  \caption{An example of $p\in L'[(b,k)\to(0,n)]$. The lower region $R'$ can be
    identified with a rotated partition contained in $B(k-n,b-1)$.}
  \label{fig:L'-path}
\end{figure}
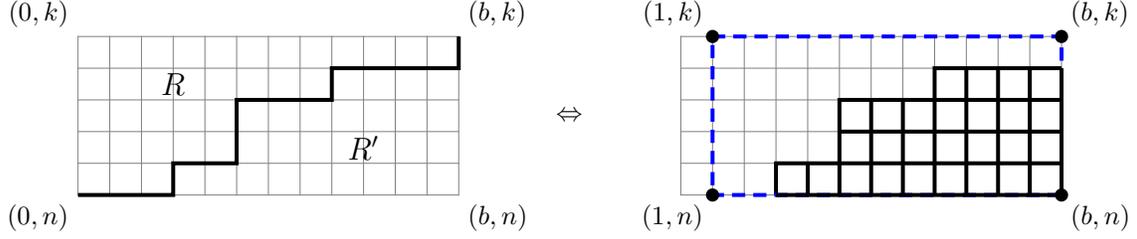

\begin{lem}\label{lem:(m,0)'}
  For $b,k\geq0$ and $m\geq1$ with $(b,k)\ne(0,0)$, we have
\[
\sum_{p\in L'[(b,k)\to(m,0)]} w(p) =  
(\epsilon q^{1-b};q)_{b-m} (-q)^{k(k-2b+2)+2(b-m)} \Qbinom{k+b-m-1}{b-m}{q^2}.
\]
\end{lem}
\begin{proof}
This can be done by the same argument as in the proof of the previous lemma.
\end{proof}

Now we can find a formula for $\bep(b,k)$. 

\begin{thm}
  For $b,k\geq0$ with $(b,k)\ne(0,0)$, we have
\begin{multline*}
\bep(b,k) = 
\sum_{i=0}^{k-1} (\epsilon q^{1-b};q)_b (-q)^{i(2k-2b-i+2)}
\Qbinom{b+i-1}{i}{q^2} \bep(0,k-i)\\
+ \sum_{i=0}^{b-1} 
(\epsilon q^{1-b};q)_i (-q)^{k(k-2b+2)+2i} \Qbinom{k+i-1}{i}{q^2}.
\end{multline*}
\end{thm}
\begin{proof}
By \eqref{eq:12'} and Lemmas~\ref{lem:(0,n)'} and \ref{lem:(m,0)'},
we have
\begin{multline*}
\bep(b,k) = \sum_{n\geq1} 
(\epsilon q^{1-b};q)_b (-q)^{(k-n)(k+n-2b+2)}
\Qbinom{b+k-n-1}{k-n}{q^2} \bep(0,n) \\
+\sum_{m\geq1} (\epsilon q^{1-b};q)_{b-m} 
(-q)^{k(k-2b+2)+2(b-m)} \Qbinom{k+b-m-1}{b-m}{q^2}. 
\end{multline*}
In the first sum the summand is zero unless $k-n\geq0$, and in the second sum
the summand is zero unless $b-m\geq0$. By replacing $k-n$ with $i$ in the first
sum and $b-m$ with $i$ in the second sum, we get the desired formula.
\end{proof}

By Theorem~\ref{thm:main2} with $\epsilon=+$ and \eqref{eq:(0,k)+}, we get a
formula for $T_k(q^{-b},q)$. 

\begin{cor}\label{cor:(-b,k)+}
For $b\geq1$ and $k\geq0$, we have
\[
T_k(q^{-b},q) = \sum_{i=0}^{b-1} (q^{1-b};q)_i (-q)^{k(k-2b+2)+2i} 
\Qbinom{k+i-1}{i}{q^2}.
\]
\end{cor}

If $b=1$ in Corollary~\ref{cor:(-b,k)+}, we get
\[
T_k(1/q,k) = (-q)^{k^2},
\]
which together with Theorem~\ref{thm:main2} implies
\[ E_n(1/q,q) = \frac{1}{(1-q)^{2n}}
\sum_{k=0}^n \left( \binom{2n}{n-k} - \binom{2n}{n-k-1} \right) (-1)^k,
\]
which is equal to $1$ if $n=0$, and $0$ otherwise.  Notice that this corresponds
to the trivial identity:
\[
 \sum_{n\geq0} E_{n}(1/q,q) x^n = 1. 
\]

By Theorem~\ref{thm:main2} with $\epsilon=-$ and \eqref{eq:(0,k)-}, we get a
formula for $T_k(-q^{-b},q)$. 

\begin{cor}\label{cor:(-b,k)-}
For $b\geq1$ and $k\geq0$, we have
\begin{multline*}
T_k(-q^{-b},q) = 
\sum_{i=0}^{k-1} (-q^{1-b};q)_b (-q)^{i(2k-2b-i+2)}
\Qbinom{b+i-1}{i}{q^2} \\
+(-q)^{k^2+2k-2kb} \sum_{i=0}^{b-1} 
(-q^{1-b};q)_i q^{2i} \Qbinom{k+i-1}{i}{q^2}.
\end{multline*}
\end{cor}

If $b=1$ in Corollary~\ref{cor:(-b,k)-}, we get
\begin{align*}
T_k(-1/q,q) &= 2\sum_{i=0}^{k-1} (-q)^{i(2k-i)} + (-q)^{k^2}\\
&= (-q)^{k^2} + 2\sum_{i=1}^{k} (-q)^{(k-i)(k+i)}\\
&=(-q)^{k^2} \sum_{i=-k}^{k} (-q)^{-i^2},
\end{align*}
which together with Theorem~\ref{thm:main2} implies the following formula for
$E_n(-1/q,q)$.

\begin{prop}\label{thm:-1/q}
We have
\[ E_n(-1/q,q) = \frac{1}{(1-q)^{2n}}
\sum_{k=0}^n \left( \binom{2n}{n-k} - \binom{2n}{n-k-1} \right)
\sum_{i=-k}^k (-q)^{i^2}.
\]
\end{prop}

We note that Proposition~\ref{thm:-1/q} was first discovered by \JV. (personal
communication).

\section{The original formula of \JV. for $E_n(q)$}
\label{sec:jv.-orig-form}

The original formula for $E_n(q)$ in \cite{JV2010} is the following:
\begin{equation}
  \label{eq:JV0}
E_{2n}(q) = \frac{1}{(1-q)^{2n}} 
 \sum_{k=0}^n \left( \binom{2n}{n-k} - \binom{2n}{n-k-1} \right)
 \sum_{i=0}^{2k} (-1)^{i+k} q^{i(2k-i)+k}, 
\end{equation}
\begin{equation}
  \label{eq:JV1}
 E_{2n+1}(q) = \frac{1}{(1-q)^{2n+1}} 
 \sum_{k=0}^n \left( \binom{2n+1}{n-k} - \binom{2n+1}{n-k-1} \right)
 \sum_{i=0}^{2k+1} (-1)^{i+k} q^{i(2k+2-i)}.  
\end{equation}

In this section we prove that \eqref{eq:seq} and \eqref{eq:tan} are equivalent
to \eqref{eq:JV0} and \eqref{eq:JV1} respectively.  By changing the index $i$
with $i+k$ in \eqref{eq:JV0} we obtain \eqref{eq:seq}.  For the second identity,
let
\[
f(k) = \frac{1}{1-q} \sum_{i=0}^{2k+1} (-1)^{i+k} q^{i(2k+2-i)}.
\]
Using Pascal's identity, we obtain
that $(1-q)^{2n} E_{2n+1}(q)$ is equal to
  \begin{align*}
&  \sum_{k=0}^n \left( \binom{2n+1}{n-k} - \binom{2n+1}{n-k-1} \right) f(k)\\
=&  \sum_{k=0}^n \left( \binom{2n}{n-k-1} - \binom{2n}{n-k-2} \right) f(k)
+\sum_{k=0}^n \left( \binom{2n}{n-k} - \binom{2n}{n-k-1} \right) f(k) \\
=&  \sum_{k=1}^{n+1} \left( \binom{2n}{n-k} - \binom{2n}{n-k-1} \right) f(k-1)
+\sum_{k=0}^n \left( \binom{2n}{n-k} - \binom{2n}{n-k-1} \right) f(k).
\end{align*}
Since $\left( \binom{2n}{n-k} - \binom{2n}{n-k-1} \right) f(k-1)=0$ when $k=0$
and $k=n+1$, we have
\begin{equation}
  \label{eq:1}
E_{2n+1}(q) = \frac{1}{(1-q)^{2n}} 
\sum_{k=0}^n \left( \binom{2n}{n-k} - \binom{2n}{n-k-1} \right) \big(f(k)+f(k-1)\big).
\end{equation}
Thus in order to get \eqref{eq:tan} it suffices to show $f(k)+f(k-1)=q^{k(k+2)}
A_k(q^{-1})$. Since
\begin{align*}
(1-q) f(k) &= \sum_{i=0}^{2k+1} (-1)^{i+k} q^{i(2k+2-i)}  
= \sum_{i=-(k+1)}^{k} (-1)^{i+2k+1} q^{(i+k+1)(k+1-i)}  \\
&= - q^{(k+1)^2}\sum_{i=-(k+1)}^{k} (-q)^{-i^2}  
= (-1)^k - q^{(k+1)^2}\sum_{i=-k}^{k} (-q)^{-i^2},
\end{align*}
we have $f(k)+f(k-1)=1$ if $k=0$, and for $k\geq1$, 
\[
f(k)+f(k-1) = - \frac{q^{(k+1)^2}}{1-q}\sum_{i=-k}^{k} (-q)^{-i^2}
- \frac{q^{k^2}}{1-q}\sum_{i=-(k-1)}^{k-1} (-q)^{-i^2},
\]
which is easily seen to be equal to $q^{k(k+2)}A_k(q^{-1})$.  Thus we get
\eqref{eq:tan}.

\section{Concluding remarks}
\label{sec:concluding-remarks}

In this paper we have found a formula for the coefficient $E_n(t,q)$ of $x^n$ in
the continued fraction
\[
\cfrac{1}{
    1- \cfrac{\qint{1}\Qint{1}{t,q} x}{
       1- \cfrac{\qint{2}\Qint{2}{t,q} x}{\cdots}}}.
\]

Since $E_n(t,q)$ is a generalization of the $q$-Euler number, it is natural to
consider a similar generalization of \eqref{eq:3}. Thus we propose the following
problem. 

\begin{problem}
Find a formula for the coefficient of $x^n$ in the following continued fraction:
\[
\cfrac{1}{
  1- \cfrac{\Qint{1}{t,q} x}{
     1- \cfrac{\Qint{2}{t,q} x}{\cdots}}}. 
\]
\end{problem}

Also, we can consider a generalization of $E_n(t,q)$ as follows.

\begin{problem}
Find a formula for the coefficient of $x^n$ in the following continued fraction:
\[
\cfrac{1}{
    1- \cfrac{\Qint{1}{y,q}\Qint{1}{t,q} x}{
       1- \cfrac{\Qint{2}{y,q}\Qint{2}{t,q} x}{\cdots}}}.
\]
\end{problem}
 
Recently Prodinger \cite{Prodinger} expressed the continued fractions (in fact
the corresponding $T$-fractions, see \cite[Lemma~6.1]{JVKim} for the relation
between $S$-fractions and $T$-fractions) in the above two problems as fractions
of formal power series when both $y$ and $t$ are equal to $q^d$ for a positive
integer $d$.  From another result of Prodinger \cite[Section~11]{Prodinger}, one
can obtain the following formula for $T_k(q^b,q)$ for a positive integer $b$:
\begin{equation}
  \label{eq:prodinger}
T_k(q^b,q) = \sum_{i=0}^b q^{\binom{i+1}2} \qbinom{b}{i}
\sum_{j=-k}^{k-i} (-1)^j q^{j^2+i(k+j)} \qbinom{k+j+b}b.
\end{equation}

\begin{problem}
  Find a direct proof of the equivalence of \eqref{eq:++b} and
  \eqref{eq:prodinger}.
\end{problem}

In the introduction we have two formulas \eqref{eq:JV_mu} and
Corollary~\ref{cor:Tk} for $E_n(t,q)$.  Using hypergeometric series Kim and
Stanton \cite{KimStanton} showed that these are equivalent and simplified to the
following formula:
\begin{equation}
  \label{eq:KS}
E_n(t,q) = \frac{1}{(1-q)^{2n}}
\sum_{k=0}^n \left( \binom{2n}{n-k} - \binom{2n}{n-k-1} \right) (-q)^k
\sum_{i=0}^{k} t^i q^{\binom{k-i}{2}} (q;q^2)_i \qbinom{k+i}{k-i}. 
\end{equation}
 
Han et al.~\cite{Han1999} introduced the polynomials $P_n^{\alpha}(x,q)$ defined
by $P_0^{\alpha}(x,q)=1$ and
\[
P_n^{\alpha}(x,q) = [x,a]_q 
\frac{[x,b]_q P_{n-1}^{\alpha}([x,c]_q,q) - [x,d]_q P_{n-1}^{\alpha}(x,q)}
{1+(q-1)x},
\]
where $\alpha=(a,b,c,d)$ is a tuple of nonnegative integers and
$[x,n]_q = xq^n + [n]_q$. They proved that
\[
\sum_{n\ge0} P_n^{\alpha}(x,q) z^n
=
\cfrac{1}{
  1-\cfrac{q^d[b-d]_q[x,a]_q z}{
    1- \cfrac{q^a[c]_q[x,b]_qz}{
      1-\cfrac{q^d[b-d+c]_q[x,a+c]_qz}{
        1-\cfrac{q^a[2c]_q[x,b+c]_qz}{
          \cdots}}}}}.
\]
One can easily check that $E_n(t,q)=P_{n}^{(0,1,2,0)} ([1]_{t,q},q)$.  Thus as a
special case of \cite[Proposition~1]{Han1999} we have
\[
\sum_{n\ge0} E_n(t,q) z^n = 
\sum_{m\ge0} \frac{tq^{2m+1} [2m]_{t,q}!}
{\prod_{i=0}^m (tq^{2i+1} + [2i+1]_{t,q}^2 z)} z^m.
\]
Using the idea in the last section of \cite{Shin2010} Zeng proved
the following formula (personal communication):
\begin{equation}
  \label{eq:zeng}
E_n(t,q) = t^{-n} 
\sum_{m=0}^n \sum_{i=0}^m (-1)^{n-i}
\frac{q^{2m-2in+i^2-n-i} [2m]_{t,q}! [2i+1]_{t,q}^{2n}}
{[2i]_q!! [2m-2i]_q!! \prod_{k=0,k\ne i}^m [2k+2i+2]_{t^2,q}},
\end{equation}
where $[2m]_{t,q}! = \prod_{i=1}^{2m} [i]_{t,q}$ and
$[2i]_q!! = \prod_{k=1}^i [2k]_q$. 

\begin{problem}
  Find a direct proof of the equivalence of \eqref{eq:KS} and \eqref{eq:zeng}.
\end{problem}

\section*{Acknowledgement}
I am grateful to Dennis Stanton for helpful discussions especially on the idea
of the proof of Lemma~\ref{lem:dist}. I also thank Matthieu \JV. for
helpful comments and Jiang Zeng for the formula \eqref{eq:zeng}.

\end{document}